\documentclass
[
    a4paper,
    DIV=10,
    abstracton
]
{scrartcl}

\usepackage{
    a4wide,
    amsmath,
    amssymb,
    amsthm,
    bm,
    graphicx,
    caption,
    subcaption,
    enumitem,
    xcolor,
    mathtools,
    authblk,
    dsfont,
    listings,
    float,
    algorithmic,
    csvsimple,
    stackengine,
    makecell,
    booktabs,
    multirow,
    bbm,
    mathrsfs,
}

\usepackage{tabularx}
\usepackage{array}

\usepackage[section]{algorithm}
\usepackage[percent]{overpic}

\usepackage[utf8]{inputenc}
\usepackage[pdffitwindow=false,
            plainpages=false,
            pdfpagelabels=true,
            pdfpagemode=UseOutlines,
            pdfpagelayout=SinglePage,
            bookmarks=false,
            colorlinks=true,
            hyperfootnotes=false,
            linkcolor=blue,
            urlcolor=blue!30!black,
            citecolor=green!50!black]{hyperref}

\newcommand{\R}{\mathbb{R}}                                     
\newcommand{\X}{\mathbb{X}}                                     
\newcommand{\Y}{\mathbb{Y}}                                     

\newcommand{\M}{\mathbb{M}}
\newcommand{\N}{\mathbb{N}}
\renewcommand{\H}{\mathbb{H}} 
\providecommand{\norm}[1]{\left\lVert #1 \right\rVert}          

\newcommand{\hl}[1]{#1}

\newcommand{\ts}{\hspace*{0.1em}} 

\newcommand\xqed[1]{\leavevmode\unskip\penalty9999 \hbox{}\nobreak\hfill \quad\hbox{#1}}
\newcommand{\exampleSymbol}{\xqed{$\triangle$}}

\DeclareMathOperator*{\argmin}{arg\,min}

\mathtoolsset{centercolon} 
\newtheorem{theorem}{Theorem}[section]
\newtheorem{corollary}[theorem]{Corollary}
\newtheorem{lemma}[theorem]{Lemma}
\newtheorem{proposition}[theorem]{Proposition}
\newtheorem{definition}[theorem]{Definition}
\theoremstyle{definition}
\newtheorem{example}[theorem]{Example}
\newtheorem{remark}[theorem]{Remark}

\definecolor{boxback}{gray}{0.95}

\graphicspath{{./pictures/}}

\title{Dimensionality Reduction of Complex Metastable Systems via Kernel Embeddings of Transition Manifolds}

\author[1]{Andreas Bittracher}
\author[1]{Stefan Klus}
\author[2,3]{\\Boumediene Hamzi}
\author[1]{Péter Koltai}
\author[1,4]{Christof Sch\"utte}

\affil[1]{Department of Mathematics and Computer Science, Freie Universit\"at Berlin, Germany}
\affil[2]{Department of Mathematics, Imperial College London, London, UK}
\affil[3]{AlFaisal University, Riyadh, KSA}
\affil[4]{\normalsize Zuse Institute Berlin, Germany}
\date{}

\begin{document}
\maketitle


\begin{abstract}
We present a novel kernel-based machine learning algorithm for identifying the low-dimensional geometry of the effective dynamics of high-dimensional multiscale stochastic systems.
Recently, the authors developed a mathematical framework for the computation of optimal reaction coordinates of such systems that is based on learning a parametrization of a low-dimensional transition manifold in a certain function space. In this article, we enhance this approach by embedding and learning this transition manifold in a reproducing kernel Hilbert space, exploiting the favorable properties of kernel embeddings.  Under mild assumptions on the kernel, the manifold structure is shown to be preserved under the embedding, and distortion bounds can be derived. This leads to  a more robust and more efficient algorithm compared to previous parametrization approaches.
\end{abstract}

\section{Introduction}
Many of the dynamical processes investigated in the sciences today are characterized by the existence of phenomena on multiple, interconnected time scales that determine the long-term behavior of the process.
Examples include the inherently multiscale dynamics of atmospheric vortex- and current formation which needs to be considered for effective weather prediction~\cite{Klein2010,Majda2003},
or the vast difference in time scales on which bounded atomic interactions, side-chain interactions, and the resulting formation of structural motifs occur in biomolecules~\cite{Freddolino2010,CaTh93,Bowman2011}.
An effective approach to analyzing these systems is often the identification of a low-dimensional observable of the system that captures the interesting behavior on the longest time scale.
However, the computerized identification of such observables from simulation data poses a significant computational challenge, especially for high-dimensional systems.

Recently, the authors have developed a novel mathematical framework for identifying such essential observables for the slowest time scale of a system~\cite{Bittracher2017}.
The method---called the \emph{transition manifold approach}---was primarily motivated by molecular dynamics, where the dynamics is typically described by a thermostated Hamiltonian system or diffusive motion in molecular dynamics landscapes. In these systems, local minima of the potential energy landscape induce \emph{metastable behavior}, which is the phenomenon that on long time scales, the dynamics is characterized by rare transitions between certain sets that happen roughly along interconnecting \emph{transition pathways}~\cite{PNAS09,SS13,towards_tpt2006}.
The sought-after essential observables should thus resolve these transition events, and are called \emph{reaction coordinates} in this context~\cite{SoEtAl96,Best2005}, a notion that we will adopt here.
Despite of its origins, the transition manifold approach is also applicable to other classes of reducible systems (which will also be demonstrated in this article).

At the heart of this approach is the insight that good reaction coordinates can be found by parametrizing a certain \emph{transition manifold} $\M$ in the function space $L^1$. This manifold has strong connections to the aforementioned transition pathway~\cite{BBS18}, but is not equivalent.
Its defining property is that, for times $\tau$ that fall between the fastest and slowest time scales, the \emph{transition density functions} with relaxation time $\tau$ concentrate around $\M$. The algorithmic strategy to parametrize $\M$ can then be summarized as
\begin{enumerate}[leftmargin=1em, itemindent=1em, topsep=0.5ex, itemsep=0ex]
\item Choose test points in the dynamically relevant regions of the state space.
\item Sample the transition densities for each test point by Monte Carlo simulation.
\item Embed the transition densities into a Euclidean space by a generic embedding function.
\item Parametrize the embedded transition densities with the help of established manifold learning techniques.
\end{enumerate}
The result is a reaction coordinate evaluated in the test points. This reaction coordinate has been shown to be as expressive as the dominant eigenfunctions of the transfer operator associated with the system~\cite{Bittracher2017}, which can be considered an ``optimal'' reaction coordinate~\cite{FGH14b,A19-1,MHP17}.
One decisive advantage of our method, however, is the ability to compute the reaction coordinate \emph{locally} (by choosing the test points), whereas with conventional methods, the inherently global computation of transfer operator eigenfunctions quickly becomes infeasible due to the curse of dimensionality. Kernel-based methods for the computation of eigenfunctions alleviate this problem to some extent \cite{SP15, KBSS18}. Nevertheless, the number of dominant eigenfunctions critically depends on the number of metastable states, which can be significantly larger than the \emph{natural} dimension of the reaction coordinate \cite{Bittracher2017}.

However, the algorithm proposed originally had several shortcomings related to the choice of the embedding function. First, in order to ensure the preservation of the manifold's topology under the embedding, the dimension of $\M$ had to be known in advance. Second, the particular way of choosing the embedding functions allowed no control over the distortion of $\M$, which may render the parametrization problem numerically ill-conditioned.

The goal of this article is to overcome the aforementioned problems by \emph{kernelizing} the transition manifold embedding. That is, we present a method to implicitly embed the transition manifold into a \emph{reproducing kernel Hilbert space} (RKHS) with a proper kernel. 
The RKHS is---depending on the kernel---a high- or even infinite-dimensional function space with the crucial property that inner products between points embedded into it can be computed by cheap kernel evaluations, without ever explicitly computing the embedding~\cite{Steinwart2008:SVM,Scholkopf2001}.
In machine learning, this so-called \emph{kernel trick} is often used to derive nonlinear versions of originally linear algorithms, by interpreting the RKHS-embedding of a data set as a high-dimensional, nonlinear transformation, and (implicitly) applying the linear algorithm to the transformed data. 
This approach has been successfully applied to methods such as \emph{principal component analysis} (PCA) \cite{Scholkopf98:KPCA}, \emph{canonical correlation analysis} (CCA) \cite{MRB01:CCA}, and \emph{time-lagged independent component analysis} (TICA) \cite{SP15}, to name but a few. 

Due to their popularity, the metric properties of the kernel embedding are well-studied~\cite{Smola07Hilbert,Fuk07,Sri10,gretton2012kernel,MFSS16}. In particular, for characteristic kernels, the RKHS is ``large'' in an appropriate sense, and geometrical information is well-preserved under the embedding. For our application, this means that distances between points on the transition manifold $\M$ are approximately preserved, and thus the distortion of $\M$ under the embedding can be bounded.
This will guarantee that the final manifold learning problem is well-posed.
Moreover, if the transformation induced by the kernel embedding is able to approximately linearize the transition manifold, there is hope that efficient \emph{linear} manifold learning methods can be used to parametrize the embedded transition manifold. 

The main contributions of this work are as follows:
\begin{enumerate}[leftmargin=1em, itemindent=1em, topsep=0.5ex, itemsep=0ex]
\item We develop a kernel-based algorithm to approximate transition manifolds and compare it with the Euclidean embedding counterpart.
\item We derive measures for the distortion of the embedding and associated error bounds.
\item We illustrate the efficiency of the proposed approach using academic and molecular dynamics examples.
\end{enumerate}
In Section~\ref{sec:reducibility criteria}, we will formalize the definition of transition manifolds and derive conditions under which systems possess such manifolds. Section~\ref{sec:Kernel-based learning of the transition manifold} introduces kernels and the induced RKHSs. Furthermore, we show that the algorithm to compute transition manifolds numerically can be written purely in terms of kernel evaluations and derive measures for the distortion of the manifold caused by the embedding into the RKHS. Numerical results illustrating the benefits of the proposed kernel-based methods are presented in Section~\ref{sec:examples} and a conclusion and future work in Section~\ref{sec:Conclusion}.

\section{Reaction coordinates based on transition manifolds}
\label{sec:reducibility criteria}

In what follows, let $\{X_t\}_{t\geq 0}$ (abbreviated as $X_t$) be a reversible, thus ergodic, stochastic process on a compact state space $\mathbb{X}\subset\mathbb{R}^n$ and $\rho$ its unique invariant density. That is, if $X_0\sim\rho$, then $X_t\sim \rho$ for all $t \ge 0$. For $x\in\X$ and $t\geq 0$, let \hl{$p^t_x \colon \X\rightarrow \R$} denote the transition density function of the system, i.e., \hl{$p^t_x$ describes the probability density at time $t$, after having started in point $x$ at time $0$.} \hl{For fixed $t$ and $x$, we will often consider $p^t_x$ as a point in the space $L^1(\X)$ (the space of absolutely integrable functions over $\X$ with respect to the Lebesgue measure), and later also other related function spaces. For the sake of clarity, we will from now on omit the argument of $L^p$ when referring to functions over $\X$.}

\subsection{Reducibility of dynamical systems}

We assume the state space dimension $n$ to be large. The main objective of this work is the identification of good low-dimensional \emph{reaction coordinates} (RCs) or \emph{order parameters} of the system. An $r$-dimensional RC is a smooth map $\xi \colon \X \rightarrow \Y$ from the full state space $ \X $ to a lower-dimensional space $\Y \subset \mathbb{R}^r$, $r\ll n$.
\hl{Loosely speaking, we call such an RC \emph{good} if on long enough time scales the projected process $\xi(X_t)$ is approximately Markovian and the dominant spectral properties of the operator describing its density evolution of $\xi(X_t)$ resemble those of $X_t$. This ensures that important long-time statistical properties such as equilibration times are preserved under projection onto the RC. We will now introduce a method to find such RCs. The explanation that the found RCs are indeed amenable to a rigorous quantitative measure of quality will be given in Section~\ref{sec:transfer operators}.}

\hl{
The method is based on a novel framework, introduced by some of the authors in~\cite{Bittracher2017}, that ties the existence of good RCs to certain geometrical properties of the family of transition densities. This framework is called the \emph{transition manifold framework}. To motivate the main idea, we first introduce the \emph{fuzzy transition manifold}:}

\begin{definition}
\label{def:fuzzy transition manifold}
Let $\tau>0$ be fixed. The set \hl{of functions}
$$
\widetilde{\mathbb{M}} := \{ p^\tau_x \mid x\in\mathbb{X}\} \hl{\subset L^1},
$$
is called the \emph{fuzzy transition manifold} of the system.
\end{definition}

\hl{The main idea behind the framework is now based on the following observation: If the system is in a certain way separable into a slowly equilibrating $r$-dimensional part and a quickly equilibrating $(n-r)$-dimensional part, and the lag time $\tau$ is chosen large enough so that the latter is essentially equilibrated but the former is not, then $\widetilde{\M}$ concentrates around an $r$-dimensional manifold in $L^1$. Further, any parametrization of this manifold is a good RC. Let us illustrate this with the aid of a simple example.}

\begin{example}
Consider the process $X_t$ to be described by overdamped Langevin dynamics
\begin{equation}
    \label{eq:Langevin dynamics}
    d X_t = -\nabla V(X_t) \ts dt + \sqrt{2/\beta} \ts dW_t,
\end{equation}
with the energy potential $V$, the inverse temperature $\beta$, and Brownian motion $W_t$. The potential depicted in Figure~\ref{fig:Transition manifold concept} possesses two metastable states, located around the local energy minima. Any realization of this system that started in one of the the metastable sets will likely remain in that set for a long time. Moreover, a realization that started outside of the wells will with high probability quickly leave the transition region and move into one of the two metastable sets. The probability of whether the trajectory will be trapped in the left or right well depends almost exclusively on the horizontal coordinate of the starting point $x$, or, in other words, the progress of $x$ along the transition path. Thus, for times $\tau$ that allow typical trajectories to move into one of the wells, the transition densities $p^\tau_x$ also depend only on the horizontal but not the vertical coordinate of $x$. This means that the fuzzy transition manifold \hl{$\widetilde{\M}$ tightly concentrates around} a one-dimensional manifold in \hl{$L^1$}. Also, any parametrization of this manifold corresponds to a parametrization of the horizontal coordinate, and thus to a good RC. \exampleSymbol
\begin{figure}
\centering
\includegraphics[width=.8\textwidth]{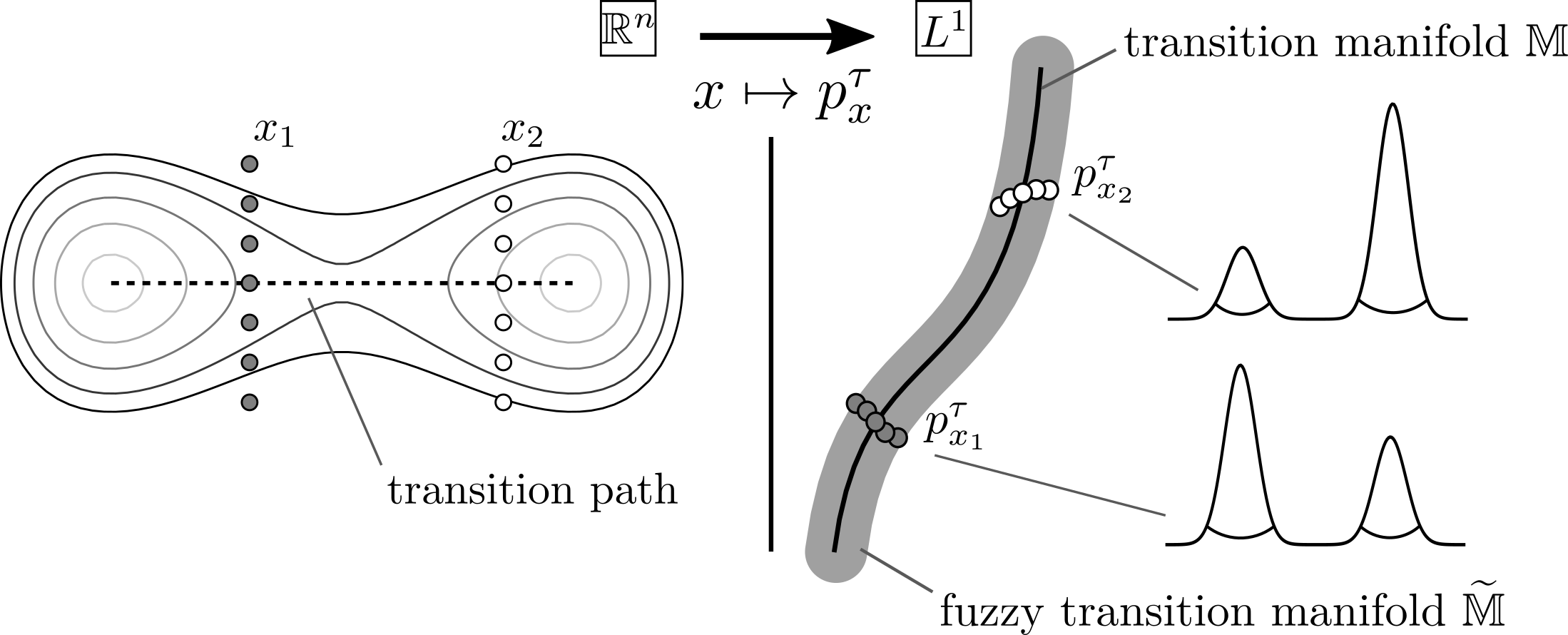}
\caption{Illustration of the transition manifold concept for metastable systems.}
\label{fig:Transition manifold concept}
\end{figure}

\end{example}

The concept of (fuzzy) transition manifolds can be made rigorous by the following definition:
\begin{definition}
\label{def:reducibility}
The process $X_t$ is called \emph{$(\varepsilon,r)$-reducible} if there exists an $r$-dimensional manifold \hl{of functions} $\mathbb{M}\subset\widetilde{\mathbb{M}}$ such that for a fixed lag time $\tau$, it holds that
\begin{equation}
\label{eq:def_reducibility}
\min_{f\in\mathbb{M}} \| f-p^\tau_x\|_{L^2_{1/\rho}} \leq \varepsilon \quad \text{for all } x\in\mathbb{X}.
\end{equation}
Any such $\M$ is called a \emph{transition manifold} of the system.
\end{definition}

Two remarks are in order:
\begin{enumerate}
\item \hl{
Note that in the above definition, the $L^2_{1/\rho}$ norm is used to measure distances, where $L^2_{1/\rho}$ is the space of (equivalence classes of) functions that are square-integrable with respect to the measure induced by the function $\frac{1}{\rho}$, and thus for $f\in L^2_{1/\rho}$,
	$$
	\|f\|_{L^2_{1/\rho}}= \left(\int_\X f(x)^2\frac{1}{\rho(x)}\ts dx\right)^{1/2}.
	$$
Closeness with respect to the $L^2_{1/\rho}$-norm instead of the $L^1$-norm is indeed a strict requirement here, as measuring the quality of a given RC will require a Hilbert space, see Section~\ref{sec:transfer operators}. Note that under appropriate assumptions on the system, it holds that $p^t_x\in L^2_{1/\rho}$ for all $x\in\X$. This will be shown in Lemma~\ref{lem:L1 estimation} and implies $\widetilde{\M}\subset L^2_{1/\rho}$, which together with the requirement $\M\subset\widetilde{\M}$ makes~\eqref{eq:def_reducibility} well-defined.}
\item The original definition of $(\varepsilon,r)$-reducibility (see \cite[Definition~4.4]{Bittracher2017}), is marginally different from the definition above: Instead of $\mathbb{M}\subset L^2_{1/\rho}$, we here require $\mathbb{M}\subset\widetilde{\mathbb{M}}\subset L^2_{1/\rho}$.
The introduction of this slightly stronger technical requirement allows us to later control a certain embedding error, see Proposition~\ref{prop:embedded fuzzy manifold}. Note that the proofs in~\cite{Bittracher2017} regarding the optimality of the final reaction coordinate are not affected by this change.
\end{enumerate}

In what follows, we always assume that the process $X_t$ is $(\varepsilon,r)$-reducible with small $\varepsilon$ and~$r\ll n$.

\begin{remark}
In addition to metastable systems, many other relevant problems possess such transition manifolds. For instance, one can show that, under mild conditions on $f$ and $g$, systems with explicit time scale separation, given by
\begin{align*}
    dX_t &= f(X_t,Y_t)dt + dW^{(1)}_t, \\
    dY_t &= \frac{1}{\varepsilon} g(X_t, Y_t)dt + \frac{1}{\sqrt{\varepsilon}}dW^{(2)}_t,
\end{align*}
also possess a transition manifold since $p^t_{(x_0,y_0)}$ essentially depends only on $x_0$ for $0 <\varepsilon\ll 1$ and $t \gg \varepsilon$. 
\end{remark}

\subsection{A measure for the quality of reaction coordinates}
\label{sec:transfer operators}

We will now \hl{present} a measure for evaluating the quality of reaction coordinates that is based on transfer operators\hl{, first derived in~\cite{Bittracher2017}}. The Perron--Frobenius operator $\mathcal{P}^t \colon L^1 \rightarrow L^1$ associated with the process $ X_t $ is defined by
$$
    \big(\mathcal{P}^t u\big)(y) = \int_\X u(x) \ts p^t_x(y) \ts dx.
$$
This operator can be seen as the push-forward of arbitrary starting densities, i.e., if $ X_0 \sim u $, then $ X_t\sim \mathcal{P}^tu $.

\hl{As $L^2_{1/\rho}\subset L^1$ (see~\cite[Remark~4.6]{Bittracher2017}) we can consider $\mathcal{P}^t$ as an operator on the inner product space $L^2_{1/\rho}$, where it has particularly advantageous properties~(see \cite{BaRo95,SchCa92,KNKWKSN18}).} Here, it is self-adjoint due to the reversibility of $X_t$. Moreover, under relatively mild conditions, it does not exhibit any essential spectrum \cite{SS13}. Hence, its eigenfunctions form an orthonormal basis of $L^2_{1/\rho}$ and the associated eigenvalues are real. Now, the significance of the dominant eigenpairs for the system's time scales is well-known~\cite{SS13}. This is the primary reason for the choice of the $L^2_{1/\rho}$-norm in Definition~\ref{def:reducibility}.

Let $ \theta_i^t $ be the eigenvalues of $\mathcal{P}^t$, sorted by decreasing absolute value, and $ \psi_i $ the corresponding eigenfunctions, where $i=0,1,\dots$. It holds that $\theta_0=1$ is independent of $t$, isolated and the sole eigenvalue with absolute value $1$. Furthermore, $\psi_0 = \rho$. The subsequent eigenvalues decrease monotonously to zero both for increasing index and time. That is,
$$
\lim_{i\rightarrow\infty} |\theta_i^t| = 0 \quad \text{and} \quad \lim_{t\rightarrow\infty} |\theta_i^t| = 0.
$$
The associated eigenfunctions $\psi_1,\psi_2,\dots$ can be interpreted as sub-processes of decreasing longevity in the following sense: Let $u\in L^2_{1/\rho}$, with $u = \sum_{i=0}^\infty \alpha_i \psi_i $, $\alpha_i\in\R$, then
$$
\mathcal{P}^t u = \sum_{i=0}^\infty \theta_i^t \alpha_i \psi_i \approx \sum_{i=0}^d \theta_i^t \alpha_i \psi_i
$$
since for the lag time $\tau>0$ as defined above, there exists an index $d\in\mathbb{N}$ such that $|\theta_i^t|\approx 0$ for all $t\geq \tau$ and all $i> d$. Hence, the major part of the information about the long-term density propagation of $X_t$ is encoded in the $d$ dominant eigenpairs.

\hl{
The operator $\mathcal{P}^t$ describes the evolution of densities of the full process $ X_t $. In order to monitor the dependence of densities on the reduced coordinate $\xi$ only, we first introduce the projection operator $\Pi_\xi \colon L^1(\X)\rightarrow L^1(\X)$,
\begin{equation}
\label{eq:projection_operator}	
\big(\Pi_\xi(u)\big)(y) = \mathbb{E}_\rho\big[ u(\mathbf{x})~\big|~\xi(\mathbf{x}) = \xi(y) \big],
\end{equation}
where the $\rho$-weighted expectation value is taken with respect to the random variable $\mathbf{x}$.
This operator is also known as the Zwanzig projection operator from statistical physics~\cite{ZHS16}. Intuitively, $\Pi_\xi$ averages a function over the individual level sets of $\xi$.

The \emph{effective transfer operator} $\mathcal{P}^t_\xi \colon L^1(\X)\rightarrow L^1(\X)$ associated with $\xi$ is then given by
$$
\mathcal{P}^t_\xi u = \Pi_\xi \big( \mathcal{P}^t( \Pi_\xi u)\big),
$$
see \cite{Bittracher2017}.
We now want to preserve the statistics of the dominant long-term dynamics of $X_t$ under the projection onto $\xi$, i.e., 
\begin{equation}
\label{eq:density transport approximation}
\mathcal{P}^t u \approx \mathcal{P}^t_\xi u,
\end{equation}
for $t \geq \tau$, where $\tau$ is some lag time that is long enough for the fast processes, associated with the non-dominant eigenpairs, to have equilibrated. A sufficient condition for \eqref{eq:density transport approximation} is
\begin{equation*}
\Pi_\xi \psi_i \approx \psi_i, \quad i=0,\ldots,d,
\end{equation*}
that is, the dominant eigenfunctions $\psi_i$ must be almost constant along the level sets of $\xi$.
This motivates the following definition of a good reaction coordinate:
}

\begin{definition}
\label{def:good reaction coordinate}
Let $(\psi_i,\theta_i^t)$ be the eigenpairs of the Perron--Frobenius operator. Let $\tau>0$ and $d\in \mathbb{N}$ such that $\theta_i^t\approx 0$ for all $i > d$ and $t\geq \tau$. We call a function $\xi\colon\X\rightarrow \mathbb{R}^r$ a \emph{good reaction coordinate} if for all $i=0,\dots,d$ there exist functions $\tilde{\psi}_i\colon\mathbb{R}^r\rightarrow \R$ such that
\begin{equation}
\label{eq:eigenfunction parametrization}
\big\| \psi_i - \tilde{\psi}_i \circ \xi \big\|_\infty \approx 0.
\end{equation}
If condition~\eqref{eq:eigenfunction parametrization} is fulfilled, we say that $\xi$ (approximately) \emph{parametrizes} the dominant eigenfunctions.
\end{definition}
\hl{For a formal evaluation of the condition~\eqref{eq:eigenfunction parametrization}, see~\cite[Corollary~3.6]{Bittracher2017}.}

\subsection{Optimal reaction coordinates}

\label{sec:optimal reaction coordinates}

We now justify why reaction coordinates that are based on parametrizations of the transition manifold $\M$ indeed fulfill 
condition~\eqref{eq:eigenfunction parametrization}. Let $Q \colon L^2_{1/\rho}\rightarrow L^2_{1/\rho}$ be the nearest-point projection onto $\M$, i.e., 
$$
{Q}(f) = \argmin_{g\in\M} \|f-g\|_{L^2_{1/\rho}}.
$$ 
Assume further that some parametrization $\gamma \colon \M\rightarrow \R^r$ of $\M$ is known, i.e., $\gamma$ is one-to-one on $\M$ and its image in $\R^r$.
\hl{Then} the reaction coordinate $\xi \colon \R^n\rightarrow\R^k$ defined by
\begin{equation}
\label{eq:optimal reaction coordinate}
\xi(x) := \big(\gamma\circ {Q}\big)(p^\tau_x)
\end{equation}
is good in the sense of Definition~\ref{def:good reaction coordinate} \hl{due to the following theorem:}

\begin{theorem}[{{\cite[Corollary~3.8]{Bittracher2017}}}]
\label{thm:reaction coordinate error}
Let the system be $(\varepsilon,r)$-reducible and $\xi$ defined as in~\eqref{eq:optimal reaction coordinate}. Then for all $i=0,\ldots,d$, there exist functions $\tilde{\psi}_i\colon\mathbb{R}^r\rightarrow \R$ such that
\begin{equation}
\label{eq:eigenfunction parametrization error}
\big\| \psi_i - \tilde{\psi}_i \circ \xi \big\|_\infty \leq \frac{\varepsilon}{|\theta_i^\tau|}.
\end{equation}
\end{theorem}

Let us add two remarks:
\begin{enumerate}
\item The choice of the $L^2_{1/\rho}$-norm in Definition~\ref{def:reducibility} is crucial for Theorem~\ref{thm:reaction coordinate error} to hold.
\item \hl{Metastable systems typically exhibit a time scale gap after the $d$ dominant eigenvalues, i.e.,
$$
\frac{\left|\theta^t_d-\theta^t_{d+1}\right|}{\left|\theta^t_{d+1}-\theta^t_{d+2}\right|} \gg 1 \quad \text{for suitably large } t>0.
$$ 
Therefore, $\tau$ can always be chosen such that $\left|\theta_{d+1}^\tau\right|$ is close to zero and $\left|\theta_{i}^\tau\right|, i=0,\ldots,d,$ is still relatively large. Consequently, the denominator in~\eqref{eq:eigenfunction parametrization error} is not too small, and thus the RC~\eqref{eq:optimal reaction coordinate} is indeed good according to Definition~\ref{def:good reaction coordinate}.}
\end{enumerate}

The main task for the rest of the paper is now the numerical computation of an (approximate) parametrization $\gamma$ of $\M$.

%

\subsection{Whitney embedding of the transition manifold}
\label{sec:whitney embedding}

One approach to find a parametrization of $\M$, proposed by the authors in~\cite{Bittracher2017}, is to first embed $\M$ into a more accessible Euclidean space and to parametrize the embedded manifold. In order to later compare it with our new method, we will briefly describe this approach here.

To construct an embedding $\mathcal{E}$ that preserves the topological structure of $\M$, without prior knowledge about $\M$, a variant of the Whitney embedding theorem can be used. It extends the classic Whitney theorem to arbitrary Banach spaces and was proven by Hunt and Kaloshin in~\cite{HuKa99}.

\hl{
\begin{theorem}[Whitney embedding theorem in Banach spaces,~\cite{HuKa99}]
\label{thm:Whitney embedding theorem}
Let $\mathbb{V}$ be a Banach space and let $\mathbb{K}\subset\mathbb{V}$ be a manifold of dimension $r$. Let $k>2r$ and let $\alpha_0 = \frac{k-2d}{k(d+1)}$. Then, for all $\alpha\in(0,\alpha_0)$, for almost every (in the sense of prevalence) bounded linear map $\mathcal{F} \colon \mathbb{V}\rightarrow\R^k$ there exists a $C>0$ such that for all $x,y\in\mathbb{K}$,
	$$
	C \left\|\mathcal{F}(x)-\mathcal{F}(y)\right\|_2^\alpha \geq \left\|x-y\right\|_\mathbb{V},
	$$
	where $\|\cdot\|_2$ denotes the Euclidean norm in $\R^k$.
	In particular, almost every $\mathcal{F}$ is one-to-one on $\mathbb{K}$ and its image, and $\mathcal{F}^{-1}\big|_{\mathcal{F}(\mathbb{K})}$ is Hölder continuous with exponent $\alpha$. 
\end{theorem}
}

\hl{
In particular, almost every such map $\mathcal{F}$ is a homeomorphism between $\mathbb{K}$ and its image in $\mathbb{R}^k$, which in short is called an \emph{embedding} of $\mathbb{K}$~(see e.g.~\cite[\S 18]{Munkres2000}). This means that the image $\mathcal{F}(\mathbb{M})$ will again be an $r$-dimensional manifold in $\mathbb{R}^k$, provided that $k>2r$.
We will apply this result to the transition manifold, i.e., $\mathbb{V}=L^2_{1/\rho}$ and $\mathbb{K}=\M$, and for simplicity restrict ourselves to the lowest embedding dimension, i.e., $k=2r+1$. Any ``randomly selected'' continuous map $\mathcal{F}:L^2_{1/\rho} \rightarrow \mathbb{R}^{2r+1}$ then is an embedding of $\M$.
}

\hl{
Unfortunately, there is no practical way to randomly draw from the space of continuous maps on $L^2_{1/\rho}$ directly. Instead of arbitrary continuous maps, we therefore restrict our considerations to maps $\mathcal{F} \colon L^2_{1/\rho}\rightarrow \mathbb{R}^{2r+1}$ of the form
\begin{align}
	\mathcal{F}(f) &:= \int_{\mathbb{X}} \eta(x') f(x') \ts dx', \label{eq:integral_embedding}
	\shortintertext{where}
\eta(x)&:= A x,\quad A\in \R^{(2r+1)\times d},~A\sim \sigma, \notag
\end{align}
where $\sigma$ is some distribution on the (finite-dimensional) space of $(2r+1)\times d$-matrices (e.g., Gaussian matrices). The linear map $\eta \colon \X\rightarrow \R^{2r+1}$, called \emph{feature map}, is bounded due to the boundedness of $\X$. Maps of the form~\eqref{eq:integral_embedding} are therefore continuous on $L^1$, and thus in particular on the subspace $L^2_{1/\rho}$.
}

\hl{
By drawing from the distribution $\sigma$ of the matrices $A$, we can effectively sample maps of form~\eqref{eq:integral_embedding}. We assume from now on that the embedding property from Theorem~\ref{thm:Whitney embedding theorem} not only holds for a prevalent subset of general continuous maps, but already for a prevalent subset of the maps of form~\eqref{eq:integral_embedding}. In other words, we assume that a randomly drawn function of form~\eqref{eq:integral_embedding} with linear $\eta$ almost surely is an embedding of $\M$. While the validity of this assumption is far from obvious for general manifolds, there is empirical evidence that for the typically ``smooth'' transition manifolds $\M$, it is indeed fulfilled. Still, this necessary restriction to the class of linear embedding functions represents a weak point of the transition manifold method that will later be solved by using kernel embeddings instead.
}

The \emph{dynamical embedding} of a point $x\in\mathbb{X}$ is then defined by 
\begin{equation}
\label{eq:dynamical embedding}
	\mathcal{E}(x) := \mathcal{F} (p^t_x) =  \int \eta(x') p^t_x(x') \ts dx'.
\end{equation}
This is the Euclidean representation of the density $p^t_x$, and the set $\{\mathcal{E}(x) \mid x\in\mathbb{X}\}\subset\mathbb{R}^{\hl{2r+1}}$ is the Euclidean representation of the fuzzy transition manifold. It again clusters around an $r$-dimensional manifold in $\R^{2r+1}$, namely the image $\mathcal{F}(\M)$ of the transition manifold under~$\mathcal{F}$:

\begin{proposition} 
\label{prop:embedded fuzzy manifold}
Let the process $X_t$ be $(\varepsilon,r)$-reducible with transition manifold $\M$, and $\mathcal{F} \colon L^2_{1/\rho}\rightarrow \R^{2r+1}$ and $\mathcal{E} \colon \R^n\rightarrow \R^{2r+1}$ defined as in~\eqref{eq:integral_embedding} and~\eqref{eq:dynamical embedding}. Then
$$
\inf_{v\in\mathcal{F}(\M)} \| v - \mathcal{E}(x) \|_\infty \leq \|\eta\|_\infty \ts \varepsilon \quad \text{for all } x\in\X.
$$
\end{proposition}

\begin{proof}
Let $x\in\X$. \hl{By the $(\varepsilon,r)$-reducibility of $X_t$ (Definition~\ref{def:reducibility}), and the fact that $\M\subset \widetilde{\M}$, i.e., $\M$ \hl{itself} consists of transition densities, there exists an $x^*\in\X$ such that $p^t_{x^*}\in \M$ and} $\|p^t_x - p^t_{x^*}\|_{L^2_{1/\rho}}\leq \varepsilon$. Thus we have
\begin{align*}
\inf_{v\in\mathcal{F}(\M)} \| v - \underbrace{\mathcal{E}(x)}_{\hl{=\mathcal{F}(p^t_x)}}\|_\infty &\leq \| \mathcal{F}(p^t_{x^*}) - \mathcal{F}(p^t_x) \|_\infty \\
&= \left \| \int_\X \eta(x') \left( p^t_{x^*}(x') - p^t_{x}(x') \right) dx' \right\|_\infty \\
&\le \|\eta\|_\infty \underbrace{\|p^t_{x^*} - p^t_x\|_{L^1}}_{\leq \|p^t_{x^*} - p^t_x\|_{L^2_{1/\rho}}}  \leq \|\eta\|_\infty \ts \varepsilon,
\end{align*}
\hl{where $\|\cdot\|_{L^1} \leq \|\cdot \|_{L^2_{1/\rho}}$ was derived in~\cite[Remark~4.6]{Bittracher2017}.}
\end{proof}

\begin{remark}
Together, Theorem~\ref{thm:Whitney embedding theorem} and Proposition~\ref{prop:embedded fuzzy manifold} guarantee at least a minimal degree of well-posedness of the embedding problem: The embedded manifold $\mathcal{F}(\M)$ has the same topological structure as $\M$, and $\mathcal{F}(\widetilde{\M})$ clusters closely around it (if $\|\eta\|_\infty$ is small). However, guarantees on the \emph{condition} of the problem cannot be made. The manifold $\M$ will in general be distorted by $\mathcal{F}$, to a degree that might pose problems for numerical manifold learning algorithms. This problem is illustrated in Figure~\ref{fig:Bad Embedding}. Such a situation typically occurs if some of the components of the embedding $\mathcal{F}$ are strongly correlated.

Additionally, the Whitney embedding theorem cannot guarantee that the \emph{fuzzy} transition manifold $\widetilde{\M}$ will be preserved under the embedding, as analytically $\widetilde{\M}$ is not a manifold. Thus, $\mathcal{F}$ is in general not injective on $\widetilde{\M}$.

\begin{figure}[h]
\centering
\includegraphics[width=.5\textwidth]{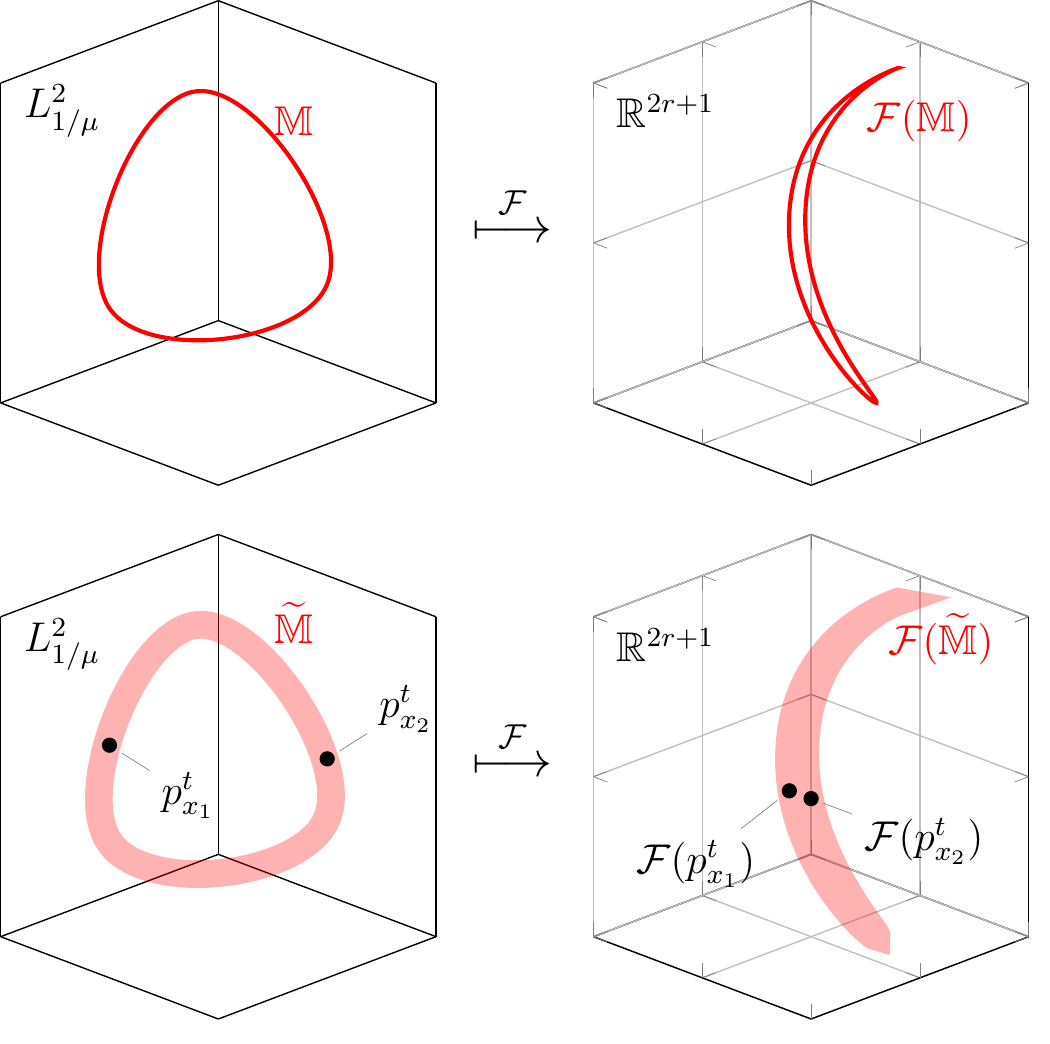}
\caption{
Illustration of the consequences of bad choices for the embedding function. While the topology of the transition manifold $\mathbb{M}$ is preserved under the embedding, the relative distances between its points may be heavily distorted (top row). Intuitively, points that lie on distant parts of the manifold might be mapped closely together. As a consequence, a manifold learning algorithm based on distances between a finite number of samples of $\mathcal{F}(\widetilde{\mathbb{M}})$ would have difficulties learning the (in this case circular) topology of $\mathbb{M}$ (bottom row).}
\label{fig:Bad Embedding}
\end{figure}

\end{remark}

\subsection{Data-driven algorithm for parametrizing the transition manifold}

\hl{
Due to the implicit definition of $\M$, the embedded transition manifold $\mathcal{F}(\M)$ is hard to analyze directly. However, as $\M\subset\widetilde{\M}$ and $\mathcal{F}(\widetilde{\M})$ concentrates $\big(\|\eta\|_\infty\varepsilon\big)$-closely around $\mathcal{F}(\M)$, one can expect that any parametrization of the dominant directions of $\mathcal{F}(\widetilde{\M})$ is also a good parametrization of $\mathcal{F}(\M)$. We now explain how $\mathcal{F}(\widetilde{\M})$ can be sampled numerically and how this sample can be parametrized.

Let $\mathbb{X}_N=\{x_1,\dots,x_N\}$ be a finite sample of state space points, which covers the ``dynamically relevant'' part of state space, i.e., the regions of $\mathbb{X}$ of substantial measure $\rho$. The exact distribution of the sampling points is not important here. If $\mathbb{X}$ is bounded or periodic, $\mathbb{X}_N$ could be drawn from the uniform distribution or chosen to form a regular grid. In practice, it often consists of a sample of the system's equilibrium measure $\rho$.

The set $\mathcal{F}\big(\{p^t_x~|~x\in \X_N\}\big)$ will serve as our sample of $\mathcal{F}(\widetilde{\M})$. Its elements can be computed numerically in the following way: Let $X_\tau(x_0,\omega)$ denote the end point of the time-$\tau$ realization of $X_t$ starting in $x_0\in\X$ and outcome $\omega\in\Omega$, where $\Omega$ is the sample space underlying the process $X_t$. For $x\in\X$, $\tau>0$ fixed as in Definition~\ref{def:reducibility} and arbitrarily chosen $\{\omega_1,\ldots,\omega_M\}\subset \Omega$, let $y^{(k)}(x) := X_{\tau}(x,\omega_k)$. In short, the $y^{(k)}(x), k=1,\ldots,n$, sample the density $p^t_x$.  In practice, the $y^{(k)}(x)$ will be generated by multiple runs of a numerical SDE solver starting in $x$ with $M$ different random seeds (``bursts of simulations'').

With the samples $y^{(k)}(x)$, we approximate $\mathcal{F}(p^t_x)$ by its Monte Carlo estimator:
\begin{align*}
\label{eq:empirical dynamical embedding}
	\mathcal{F}(p^t_x) = \int\eta(x') p^t_x(x')\ts dx' &\approx \underbrace{\frac{1}{M}\sum_{k=1}^M\eta(y^{(k)}(x))}_{\textstyle =:\widehat{\mathcal{E}}(x)}.
\end{align*}
Due to Proposition~\ref{prop:embedded fuzzy manifold}, the point cloud $\mathcal{F}\big(\{p^t_x~|~x\in\X_N\}\big)$, and for a large enough burst size $M$ also its empirical estimator $\widehat{\mathcal{E}}\big(\X_N\big)$, then clusters around the $r$-dimensional manifold $\mathcal{F}(\M)$ in $\mathbb{R}^{2r+1}$. 

Parametrizing $\widehat{\mathcal{E}}\big(\X_N\big)$, i.e., finding the dominant nonlinear directions in this point cloud in $\R^{2r+1}$, now can be accomplished by a variety of classical manifold learning methods. We assume that we have a method at our disposal that is able to discover the underlying $r$-dimensional manifold within the point cloud $\widehat{\mathcal{E}}\big(\X_N\big)$, and assign each of the points $\{\widehat{\mathcal{E}}(x)~|~x\in\X_N\}$ a value $\tilde{\gamma}\big(\widehat{\mathcal{E}}(x)\big)\in \R^r$ according to its position on that manifold. For examples of such algorithms see Section~\ref{sec:kerneltrick}. Hence, $\tilde{\gamma} \colon \widehat{\mathcal{E}}(\X_N)\rightarrow \R^k$ can be seen as an approximate parametrization of $\mathcal{F}(\M)$, defined however only at the points $\widehat{\mathcal{E}}(\X_N)$. Any parametrization of $\mathcal{F}(\M)$ in turn corresponds to a parametrization of $\M$, due to $\mathcal{F}$ being an embedding. Finally, any parametrization of $\M$ corresponds to a good reaction coordinate due to Theorem~\ref{thm:reaction coordinate error}. Thus, the map $\xi(x) \colon \X_N\rightarrow \R^r$,
$$
\xi(x) := \tilde{\gamma}\big(\widehat{\mathcal{E}}(x)\big),
$$
forms a good reaction coordinate. Note however that it is only defined on the sample points~$\X_N$.

The strategy of computing reaction coordinates by embedding densities sampled from $\widetilde{\M}$ into $\R^{2r+1}$ by a random linear map and learning a parametrization of the embedded manifold was first presented in~\cite{Bittracher2017}. The following algorithm summarizes the overall procedure:
}

\begin{algorithm}
\floatname{algorithm}{Algorithm}
\caption{\hl{Reaction coordinate computation based on Whitney embeddings.}}
\label{algo:WhitneyRC}
\begin{algorithmic}[1]
\REQUIRE{Transition manifold dimension $r$, intermediate lag time $\tau$, \hl{matrix distribution $\sigma$}.}
\STATE Choose test points $\X_N=\{x_1,\ldots,x_N\}$ that cover the relevant parts of state space.
\STATE \hl{Randomly draw a matrix $A\in\R^{(2r+1)\times d}$ from $\sigma$. Define the map $\eta \colon x\mapsto Ax$.}
\hl{\FOR {$i=1,\ldots,N$}
	{\FOR {$l=1,\ldots,M$}
		\STATE {Simulate trajectory of length $\tau$ with new random seed. Let the end point be denoted by $y_i^{(l)}$.}
		\ENDFOR
		}
\ENDFOR}
\STATE Compute the embedded empirical densities as
$
z_i \gets \frac{1}{M} \sum_{j=1}^M \eta\big(y_i^{(j)}\big) \in \mathbb{R}^{2r+1}.
$
\STATE Apply a \hl{nonlinear} manifold learning algorithm to $\{z_i~|~i=1,\ldots,N\}$. Let $\widetilde{\gamma}(z_i)\in \R^r$ denote the resulting parametrization of the embedded test points.
\ENSURE An $r$-dimensional reaction coordinate evaluated at the test points:
\begin{equation*}
    \setlength{\abovedisplayskip}{0pt}
    \setlength{\belowdisplayskip}{0pt}
    \xi(x_i) := \widetilde{\gamma}(z_i),\quad i=1,\dots,N.
\end{equation*}
\end{algorithmic}
\end{algorithm}

\section{Kernel-based parametrization of the transition manifold}
\label{sec:Kernel-based learning of the transition manifold}

The approach described above for learning a parametrization of the transition manifold $\M$ by embedding it into Euclidean spaces requires a priori knowledge of the dimension of $\M$.
Also, more importantly, $\M$ might be strongly distorted by the embedding \hl{$\mathcal{F}$}, as described in Section~\ref{sec:whitney embedding}. The kernel-based parametrization, which is the main novelty of this work, will address both of these shortcomings by embedding $\M$ into reproducing kernel Hilbert spaces.

\subsection{Kernel reformulation of the embedding algorithm}
\label{sec:kerneltrick}

\hl{Manifold learning} algorithms that can be used in Algorithm~\ref{algo:WhitneyRC} include Diffusion Maps~\cite{CKLMN08}, Multidimensional Scaling~\cite{Young2013,Kruskal1964A}, and Locally Linear Embedding \cite{Roweis2000}. These, and many others, require only a notion of distance between pairs of data points. In our case, this amounts to the Euclidean distances between embedded points, i.e., $\big\| {\mathcal{E}}(x_i) - {\mathcal{E}}(x_j)\big\|_2$, which can be computed by the Euclidean inner products $\big\langle {\mathcal{E}}(x_i),{\mathcal{E}}(x_j)\big\rangle$, as
\begin{equation*}
	\big\|\mathcal{E}(x_i) - \mathcal{E}(x_j)\big\|_2^2  = \big\langle {\mathcal{E}}(x_i),{\mathcal{E}}(x_i)\big\rangle - 2\big\langle {\mathcal{E}}(x_i),{\mathcal{E}}(x_j)\big\rangle + \big\langle {\mathcal{E}}(x_j),{\mathcal{E}}(x_j)\big\rangle.
\end{equation*}
Other compatible algorithms such as Principal Component Analysis are based directly on the inner products. The inner products can be written as
\begin{equation*}
	\big\langle {\mathcal{E}}(x_i),{\mathcal{E}}(x_j)\big\rangle = \iint \big\langle \eta(y_i),\eta(y_j)\big\rangle \ts p^t_{x_i}(y_i) \ts p^t_{x_j}(y_j) \ts dy_i \ts dy_j,
\end{equation*}
and the empirical counterpart is
\begin{equation*}
	\big\langle \widehat{\mathcal{E}}(x_i),\widehat{\mathcal{E}}(x_j)\big\rangle = \frac{1}{M^2}\sum_{l_1,l_2=1}^M \big\langle \eta(y_i^{(l_1)}),\eta(y_j^{(l_2)})\big\rangle.
\end{equation*}
However, rather than \emph{explicitly} computing the inner product between the features on the right-hand side, we now assume that it can be computed \emph{implicitly} by using  a \emph{kernel function} $k \colon \X\times\X\rightarrow \mathbb{R}$, i.e.,
\begin{equation}
\label{eq:kernel assumption}
 k(y_i,y_j) = \big\langle \eta(y_i),\eta(y_j)\big\rangle.
\end{equation}
This assumption, called the \emph{kernel trick}, is commonly used to avoid the costly computation of inner products between high-dimensional features.
However, instead of defining the kernel $k$ based on previously chosen features, one typically considers kernels that implicitly define high- and possibly infinite-dimensional feature spaces. In this way, we are able to avoid the choice of \hl{the feature map $\eta$} altogether.

Kernels with this property span a so-called \emph{reproducing kernel Hilbert space}:
\begin{definition}[Reproducing kernel Hilbert space~\cite{Scholkopf2001}]
Let $k\colon\mathbb{X}\times\mathbb{X}\rightarrow\mathbb{R}$ be a positive definite function. \hl{A Hilbert space} $\mathbb{H}$ of functions $f\colon\mathbb{X}\rightarrow\mathbb{R}$, together with the corresponding inner product $\langle\cdot,\cdot\rangle_\mathbb{H}$ \hl{and norm $\|\cdot\|_\H=\sqrt{\langle\cdot,\cdot\rangle_\mathbb{H}}$} which fulfills
\begin{enumerate}[label=(\roman*)]
\item $\mathbb{H}= \overline{\operatorname{span}\{k(x,\cdot) \mid x\in\mathbb{X}\}}$, and
\item $\langle f,k(x,\cdot)\rangle_\mathbb{H} = f(x)$ for all $f\in\mathbb{H}$ \label{itm:reproducing property}
\end{enumerate}
is called the \emph{reproducing kernel Hilbert space (RKHS)} associated with the \emph{kernel} $k$.
\end{definition}

\hl{Here, $\overline{A}$ denotes the completion of a set $A$ with respect to $\|\cdot\|_\H$.} Requirement~\ref{itm:reproducing property} implies that 
\begin{equation}
\label{eq:canonical featuremap}
\langle k(x,\cdot),k(x',\cdot)\rangle_{\mathbb{H}} = k(x,x') \quad \text{for all } x,x' \in \mathbb{X}.
\end{equation}
\hl{The inner product between general functions $f,g\in \operatorname{span}\left\{k(x,\cdot)~|~x\in\X\right\}$ can therefore be expressed as the weighted sum of kernel evaluations: Let 
$$
f=\sum_{i} \alpha_i\ts k(x_i,\cdot),\quad g=\sum_{j} \beta_j\ts k(x'_j,\cdot),
$$
where the selection of points $x_i,x'_j$ depends on $f$ and $g$, respectively. Then
$$
\left\langle f,g\right\rangle_\mathbb{H} = \sum_{i,j} \alpha_i \beta_j\ts k(x_i,x'_j).
$$
For functions on the boundary of $\operatorname{span}\left\{k(x,\cdot)~|~x\in\X\right\}$, the inner product is constructed by the usual limit procedure.
}

The map $\eta \colon x\mapsto k(x,\cdot)$ can be regarded as a function-valued feature map (the so-called \emph{canonical feature map}). However, each positive definite kernel is guaranteed to also possess a feature map of at most countable dimension\hl{:}

\begin{theorem}[{Mercer's theorem~\cite{Mercer1909}}] \label{thm:Mercer}
Let $k$ be a positive definite kernel and $\nu$ be a finite Borel measure with support $\mathbb{X}$. Define the integral operator $\mathcal{T}_k \colon L^2_\nu\rightarrow L^2_\nu$ by
\begin{equation}
\label{eq:integral operator}
\mathcal{T}_k f = \int k(\cdot,x)f(x) \ts d\nu(x).
\end{equation}
Then there is an orthonormal basis $\{\sqrt{\lambda_i} \ts \varphi_i\}$ of $\,\mathbb{H}$ consisting of eigenfunctions $ \varphi_i $ of $\mathcal{T}_k$ rescaled with the square root of the corresponding nonnegative eigenvalues $ \lambda_i $ such that
\begin{equation}
\label{eq:mercer representation}
k(x,x') = \sum_{i=0}^\infty \lambda_i\varphi_i(x)\varphi_i(x') \quad \text{for all } x,x'\in\mathbb{X}.
\end{equation}
\end{theorem}
The above formulation of Mercer's theorem has been taken from~\cite{MFSS16}. The \emph{Mercer features} $\eta_i := \sqrt{\lambda_i} \ts \varphi_i$ thus fulfill~\eqref{eq:kernel assumption} for their corresponding kernel. 
\hl{The usage of the same symbol $\eta$ as for the linear feature map from Section~\ref{sec:whitney embedding} is no coincidence, as the Mercer features will again serve the purpose to observe certain features of the full system. In what follows, $\eta(x)$ will always refer to the vector (or $\ell^2$ sequence) defined by the Mercer features.} If not stated otherwise, $ \nu $ will be the standard Lebesgue measure.

\begin{example}
\label{ex:kernels}
Examples of commonly used kernels are:
\begin{enumerate}
\item Linear kernel: $k(x,x') = x^\top x'$. One sees immediately that~\eqref{eq:kernel assumption} is fulfilled by choosing $\eta_i(x) = x_i$, $i=1,\ldots,n$ (also spanning the Mercer feature space).
\item Polynomial kernel of degree $ p $: $k(x,x') = (x^\top x' + 1)^p$. It can be shown that the Mercer feature space is spanned by the monomials in $x$ up to degree $p$.
\item Gaussian kernel: 
$
k(x,x') = \exp\left(-\frac{1}{\sigma}\norm{x-x'}_2^2\right),
$
where $ \sigma > 0 $ is called the \emph{bandwidth} of the kernel. Let $p\in\mathbb{N}$ and $\mathbf{p} = (p_1,\dots,p_n)$ with $p_1+\ldots+p_n=p$ be a multi-index.
The Mercer features of $k$ then take the form
$$
\eta_{\mathbf{p}}(x) = e_{p_1}(x_1) \cdots e_{p_n}(x_n),
$$
see \cite{Steinwart2008:SVM}, where
\begin{equation}
    e_{p_i}(x_i) = \sqrt{\frac{2^{p_i}}{\sigma^{2{p_i}}{p_i}!}} \, x_i^{p_i} \, \exp\left(-\tfrac{1}{\sigma^{2}}x_i^2\right). \tag*{\exampleSymbol}
\end{equation}
\end{enumerate}
\end{example}

Let $\mathcal{F}_k$ denote the density embedding based on the Mercer features of the kernel $k$, i.e.,
\begin{equation}
\label{eq:kernel dynamical embeddings}
\big(\mathcal{F}_k(p^t_x)\big)_i := \int \eta_i(x') p^t_x(x')\ts dx',\quad i=0,1,2,\ldots,
\end{equation}
and let $\mathcal{E}_k(x):= \mathcal{F}_k(p^t_x)$.
The amount of information about $p^t_x$ preserved by  the embedding $\mathcal{F}_k$ depends on the choice of the kernel $k$. For the first two kernels in Example~\ref{ex:kernels}, the information preserved has a familiar stochastic interpretation~(see, e.g., \cite{MFSS16,Schoelkopf2015,Sri10}):
\begin{enumerate}
\item Let $k$ be the linear kernel. Then
$$
\left\| \mathcal{F}_k(p^t_{x_1}) - \mathcal{F}_k(p^t_{x_2})\right\|_\hl{2}=0\quad \Longleftrightarrow\quad \int p^t_{x_1}(y) \ts dy = \int p^t_{x_2}(y) \ts dy,
$$
i.e., the means of $p^t_{x_1}$ and $p^t_{x_1}$ coincide.
\item Let $k$ be the polynomial kernel of degree $p>1$. Then
$$
\left\| \mathcal{F}_k(p^t_{x_1}) - \mathcal{F}_k(p^t_{x_2})\right\|_\hl{2}=0\quad \Longleftrightarrow\quad \mathbf{m}_i\big(p^t_{x_1}\big) = \mathbf{m}_i\big(p^t_{x_2}\big), \quad i=1,\dots,p,
$$
i.e., the first $p$ moments $\mathbf{m}_i$ of $p^t_{x_1}$ and $p^t_{x_1}$ coincide. 
\end{enumerate}

\begin{remark}
In practice, comparing the first $p$ moments often is enough to sufficiently distinguish the transition densities that constitute the transition manifold. However, densities that differ only in higher moments cannot be distinguished by $\mathcal{F}_k$, which means that for the above two kernels, $\mathcal{F}_k$ is not injective \hl{on $\M$}. Therefore $\mathcal{F}_k$ does not belong to the prevalent class of maps that is at the heart of the Whitney embedding theorem~\ref{thm:Whitney embedding theorem}. We can therefore not utilize the Whitney embedding theorem to argue that the topology of $\M$ is preserved under~$\mathcal{F}_k$.
Instead, in Section~\ref{sec:kernel mean embedding}, we will use a different argument to show that the embedding is indeed injective for the Gaussian kernel (and others).
\end{remark}

Still, by formally using the Mercer dynamical embedding $\mathcal{E}_k$ in~\eqref{eq:dynamical embedding} (abusing notation if there are countably infinitely many such features), and using the kernel trick, we can now reformulate Algorithm~\ref{algo:WhitneyRC} as a \emph{kernel-based} method that does not require the explicit computation of any feature vector. This is summarized in Algorithm~\ref{algo:KernelRC}.

\begin{algorithm}
\floatname{algorithm}{Algorithm}
\caption{Kernel-based computation of the reaction coordinate.}
\label{algo:KernelRC}
\begin{algorithmic}[1]
\REQUIRE{Kernel $k\colon\X\times\X\rightarrow \R$, intermediate lag time $\tau$.}
\STATE Choose test points $\X_N=\{x_1,\ldots,x_N\}$ that cover the relevant parts of state space.
\hl{\FOR {$i=1,\ldots,N$}
	{\FOR {$l=1,\ldots,M$}
		\STATE {Simulate trajectory of length $\tau$ with new random seed. Let the end point be denoted by $y_i^{(l)}$.}
		\ENDFOR
		}
\ENDFOR}
\STATE Compute the kernel matrix $K\in\R^{N\times N}$:
\begin{equation*}
    \setlength{\abovedisplayskip}{0pt}
    \setlength{\belowdisplayskip}{0pt}
    K_{ij} = \frac{1}{M^2} \sum_{l_1,l_2=1}^M k(y_i^{(l_1)},y_j^{(l_2)}).
\end{equation*}
\STATE Compute the distance matrix $D\in\R^{N\times N}$:
\begin{equation*}
    \setlength{\abovedisplayskip}{0pt}
    \setlength{\belowdisplayskip}{-2ex}
    D_{ij} = K_{ii} + K_{jj} -2 K_{ij}.
\end{equation*}
\STATE Apply a distance-based manifold learning algorithm to the distance matrix $D$. Denote the resulting parametrization of the underlying $i$-th element by $\widetilde{\gamma}_i\in \R^r$.
\ENSURE An $r$-dimensional reaction coordinate evaluated at the test points:
\begin{equation*}
    \setlength{\abovedisplayskip}{0pt}
    \setlength{\belowdisplayskip}{0pt}
    \xi(x_i) := \widetilde{\gamma}_i,\quad i=1,\dots,N.
\end{equation*}
\end{algorithmic}
\end{algorithm}

\subsection{Condition number of the kernel embedding}
\label{sec:kernel mean embedding}

We will now investigate to what extent the kernel embedding preserves the topology and geometry of the transition manifold.

\subsubsection{Kernel mean embedding}

We derived the kernel-based algorithm by considering the embedding $\mathcal{F}_k$ of the transition manifold into the image space of the Mercer features in order to highlight the similarity to the Whitey embedding based on randomly drawn features. Of course, the Mercer features never had to be computed explicitly.

However, in order to investigate the quality of this embedding procedure it is advantageous to consider a different, yet equivalent embedding map: The transition manifold can be directly embedded into the RKHS by means of the \emph{kernel mean embedding} operator.

\begin{definition}
\label{def:kernel mean embedding}
Let $k$ be a positive definite kernel and $\H$ the associated RKHS. Let $p$ be a probability density over $\mathbb{X}$. Define the \emph{kernel mean embedding} of~$p$ by
\begin{align*}
\mu(p) &:= \int_\mathbb{X}k(x,\cdot)p(x) \ts dx
\intertext{and the \emph{empirical kernel mean embedding} by}
\hat{\mu}(p) &:= \frac{1}{m}\sum_i k(x_i,\cdot)\quad \text{with}\quad \{x_1,\ldots,x_m\}\sim p.
\end{align*}
\end{definition}

Note that $\mu(p)$ and $\hat{\mu}(p)$ are again elements of $\H$ and that for $ \nu $ in \eqref{eq:integral operator} being the Lebesgue measure we obtain $ \mu(p) = \mathcal{T}_k \ts p $.
Further, one sees that
$$
\big\langle \mathcal{F}_k(p^t_{x_1}),\mathcal{F}_k(p^t_{x_2}) \big\rangle = \left\langle \mu(p^t_{x_1}), \mu(p^t_{x_2})\right\rangle_\H,
$$
\hl{where the inner product $\langle\cdot,\cdot\rangle$ refers to the Euclidean inner product or the inner product in $\ell^2(\mathbb{N}_0)$, dependent on whether $\mathcal{F}_k(p)$ is finite or countably infinite.}
Thus, for investigating whether the embedding $\mathcal{F}_k$ preserves distances or inner products between densities, we can equivalently investigate the embedding $\mu$.
This is advantageous as injectivity and isometry properties of the kernel mean embedding are well-studied.

\subsubsection{Injectivity of the kernel mean embedding}

A first important result is that $k$ can be chosen such that $\mu$ is injective. Such kernels are called \emph{characteristic}~\cite{Fuk07}. In \cite{Sri10}, several conditions for characteristic kernels are listed, including the following:

\begin{theorem}[{\cite[Theorem~7]{Sri10}}]
\label{thm:requirement characteristic}
The kernel $k$ is characteristic if for all $f\in L_2,~f\neq 0$ it holds that
\begin{equation}
\label{eq:requirement characteristic}
\int_\X\int_\X k(x,x') f(x) f(x') \ts dx \ts dx' > 0.
\end{equation}
\end{theorem}
Condition~\eqref{eq:requirement characteristic} is known as the \emph{Mercer condition}, which is, for example, fulfilled by the Gaussian kernel from Example~\ref{ex:kernels}. The Mercer features of such a kernel are particularly rich.

\begin{theorem}
\label{thm:Mercer L2 basis}
Assume that the kernel satisfies the Mercer condition~\eqref{eq:requirement characteristic}. Then the eigenfunctions $ \{ \psi_i \} $ of $ \mathcal{T}_k $ form an orthonormal basis of $ L^2(\nu) $.
\end{theorem}
For more details, see, e.g., \cite{Scholkopf2001,Steinwart2008:SVM}. It is easy to see that for kernels fulfilling~\eqref{eq:requirement characteristic}, $\mu$ as a map from $L^2$ to $\H$ is Lipschitz continuous:
\begin{lemma}
\label{lem:locality preserving}
Let $k$ be a characteristic kernel with Mercer eigenvalues $\lambda_i$, $i\in\N_0$.
Then $\mu\colon L^2\rightarrow \H$ is Lipschitz continuous with constant
\begin{equation}
\label{eq: requirement continuous}
c := \sqrt{\lambda_0}.
\end{equation}
\end{lemma}
\begin{proof}
As $\mu$ is linear, it suffices to show that $\|\mu(f)\|_\H\leq c \|f\|_2$ for all $f\in L_2$. We obtain
\begin{align*}
\|\mu(f)\|_\H^2 &= \left\langle\mu(f),\mu(f)\right\rangle_\H \\
&= \left\langle \int_\X f(x) k(x,\cdot) \ts dx, \int_\X f(x) k(x,\cdot) \ts dx \right\rangle_\H \\
&= \int_\X \int_\X f(x)f(y) k(x,y) \ts dx \ts dy,
\end{align*}
where~\eqref{eq:canonical featuremap} was used in the last line. By expanding $k$ into its Mercer features via~\eqref{eq:mercer representation}, this becomes
\begin{align*}
\|\mu(f)\|_\H^2 &= \int_\X\int_\X f(x) f(y) \left( \sum_{i\in\N} \lambda_i \varphi_i(x) \varphi_i(y)\right) \! dx \ts dy \\
&= \sum_{i\in\N} \lambda_i \langle f,\varphi_i\rangle_{L^2}^2.
\end{align*}
By Theorem~\ref{thm:Mercer L2 basis}, the $\varphi_i$ form an orthonormal basis of $L^2$, and thus
\begin{align*}
\|\mu(f)\|_\H^2 &\leq \lambda_0 \ts \|f\|_{L^2}^2. \qedhere
\end{align*}
\end{proof}

Thus, if the kernel is characteristic, the structure of the TM and the fuzzy TM are qualitatively preserved under the embedding.

\begin{corollary}
\label{cor:embedded core TM}
Let $k$ be a characteristic kernel and let $X_t$ be $(\varepsilon,r)$-reducible.
Then $\mu(\M)$ is an $r$-dimensional manifold in $\H$, and for all $x\in\X$ it holds that
$$
\inf_{g\in\mu(\M)} \|g-\mu(p^t_x)\|_\H \leq \sqrt{\lambda_0} \ts \|\sqrt{\rho}\|_\infty \ts \varepsilon.
$$
\end{corollary}
\begin{proof}
By Lemma~\ref{lem:locality preserving}, the map $\mu \colon L^2\rightarrow \H$ is continuous (and furthermore injective), and thus $\mu(\M)$ is again an $r$-dimensional manifold. For $x\in\X$, consider now any $f\in L^2_{1/\rho}$ with $\|f-p^t_x\|_{L^2_{1/\rho}}\leq \varepsilon$. For $g:= \mu(f)$, we then get
\begin{align*}
\big\|g - \mu(p^t_x)\big\|_\H &= \big\|\mu\big(f - p^t_x\big)\big\|_\H 
\shortintertext{which by Lemma \ref{lem:locality preserving} is}
&\leq \sqrt{\lambda_0} \ts \|\sqrt{\rho}\|_\infty \ts \|f-p^t_x\|_{L^2_{1/\rho}} \\
&\leq \sqrt{\lambda_0} \ts \|\sqrt{\rho}\|_\infty \ts \varepsilon. \qedhere
\end{align*}
\end{proof}

\begin{remark}
This result should be seen as an analogue to Proposition~\ref{prop:embedded fuzzy manifold} for the Whitney-based TM embedding.
In short, for characteristic kernels, the injectivity and continuity of $\mu$ guarantee that the image of $\M$ under $\mu$ is again an $r$-dimensional manifold in $\H$, and Corollary~\ref{cor:embedded core TM} guarantees that the embedded fuzzy transition manifold $\mu(\widetilde{\M})$ still clusters closely around $\mu(\M)$ (if $\sqrt{\lambda_0}$ and $\|\sqrt{\rho}\|_\infty$ in Corollary~\ref{cor:embedded core TM} are small).
This again guarantees a minimal degree of well-posedness of the problem.
\end{remark}

\subsubsection{Distortion under the kernel mean embedding}
\label{sec:distortion}

Unlike the Whitney embedding, the kernel embedding now allows us to derive conditions under which the distortion of $\M$ is bounded. We have to show that the $L^2_{1/\rho}$-distance between points on $\M$ is not overly decreased or increased by the kernel mean embedding. To formalize this, we consider measures for the manifold's internal distortion, following the notions of metric embedding theory~\cite{Abraham2011}.
\hl{
We call the embedding \emph{well-conditioned} if both the
\begin{equation}
\label{eq:compression factor}
\text{contraction: } \sup_{\substack{p,q\in\M \\ q\neq p}} \frac{\|p-q\|_{L^2_{1/\rho}}}{\|\mu(p)-\mu(q)\|_\H}
\quad
\text{and the expansion: }
 \sup_{\substack{p,q\in\M \\ q\neq p}} \frac{\|\mu(p)-\mu(q)\|_\H}{\|p-q\|_{L^2_{1/\rho}}}
\end{equation}
are small (close to one). Here, $\mu$ denotes the embedding corresponding to a characteristic kernel.


Due to the Lipschitz continuity of $\mu$ (see Lemma~\ref{lem:locality preserving}) and $\|\cdot \|_{L^2} \leq \|\sqrt{\rho}\|_\infty \ts \|\cdot\|_{L^2_{1/\rho}}$, we have 
\begin{align}
\frac{\|\mu(p) - \mu(q)\|_\H}{\|p-q\|_{L^2_{1/\rho}}} \leq \sqrt{\lambda_0}\ts \|\sqrt{\rho}\|_\infty,
\label{eq:expansion bound}
\end{align}
thus bounding the expansion.
}

\hl{
\paragraph{Contraction bound: regularity requirement.}

Unfortunately, it is not possible even for characteristic kernels to derive a bound for the contraction that holds uniformly for all $p,q\in L^2_{1/\rho}$, as the following proposition shows. Nevertheless, we will be able to give reasonable bounds under some regularity- and dynamic assumptions, \eqref{eq:assumption Mercer coefficients} and~\eqref{eq:transition_density_decomposition}, respectively.

\begin{proposition}[Unbounded inverse embedding]
\label{prop:small distance}
Assume the kernel embedding operator $\mu$ has absolutely bounded orthonormal eigenfunctions $\varphi_i$ with corresponding nonnegative eigenvalues $\lambda_i$ (arranged in nonincreasing order). Assume $\lim_{i\rightarrow \infty} \lambda_i = 0 $. Then there exist functions $p, q\in L^2_{1/\rho}$ such that 
$$
\frac{\|p-q\|_{L^2_{1/\rho}}}{\|\mu(p)-\mu(q)\|_\H} > \frac{1}{\varepsilon}
$$
for any arbitrarily small $\varepsilon>0$.
\end{proposition}
}
\begin{proof}
See Appendix~\ref{sec:proofs}.
\end{proof}
The assumptions of Proposition~\ref{prop:small distance} are fulfilled for example for the Gaussian kernel. A similar, but non-quantitative result has been derived in~\cite[Theorem~19]{Sri10}. The idea behind its proof and the proof of Proposition~\ref{prop:small distance} is that, if $p$ and $q$ vary only in higher eigenfunctions $\varphi_i$ of the embedding operator $\mu$ (see also Theorem~\ref{thm:Mercer}), the $\H$-distance can become arbitrarily small.
If, however, we can reasonably restrict our considerations to the subclass of functions whose variation in the higher $\varphi_i$ is small compared to the variation in the lower $\varphi_i$, a favorable bound can be derived. Let the expansion of $h=p-q$ be given by
$$
h = \sum_{i=0}^\infty \tilde{h}_i \varphi_i
$$
with the sequence $(\tilde{h}_0,\tilde{h}_1,\ldots)\in \ell^2$.
Now, for any $i_\text{max}\in\N$ such that there exists an index $i\leq i_\text{max}$ with $\tilde{h}_i\neq 0$, define the factor
\begin{equation}
\label{eq:modefactor}
c(h,i_\text{max}) := 1+\frac{\sum_{i=i_\text{max}+1}^\infty \tilde{h}_i^2}{\sum_{i=0}^{i_\text{max}} \tilde{h}_i^2}.
\end{equation}
This factor bounds the contribution of the higher Mercer eigenfunctions to $h$ by the contribution of the lower ones, hence it is a \emph{regularity bound}:
$$
\sum_{i=0}^\infty \tilde{h}_i^2 = c(h,i_\text{max}) \cdot \sum_{i=0}^{i_\text{max}} \tilde{h}_i^2.
$$

Thus, for an individual $h$, we can bound the distortion of the $L^2$-norm under $\mu$ with the help of $c(h,i_\text{max})$.

\begin{lemma}
\label{lem:stretchingfactor}
Let $h\in L^2$, $i_\text{max}\in\mathbb{N}$, and $c(h,i_\text{max})$ be defined as in \eqref{eq:modefactor}. Then
$$
\|\mu(h)\|_\H \geq \sqrt{\frac{\lambda_{i_\text{max}}}{c(h,i_\text{max})}}\|h\|_{L^2}.
$$
\end{lemma}
\begin{proof}
See Appendix~\ref{sec:proofs}.
\end{proof}

\hl{
We from now on make the assumption that for every index $i_\text{max}$ there exists a constant $c^*_{i_\text{max}}>0$ such that
\begin{equation}
\label{eq:assumption Mercer coefficients}
c(p^\tau_{x_1}-p^\tau_{x_2},i_\text{max}) \leq c^*_{i_\text{max}}
\end{equation}
for all $p^\tau_{x_1},p^\tau_{x_2}\in \M$. 
The existence and form of this constant strongly depends on the shape of the Mercer eigenfunctions, hence the kernel. However, we motivate the existence of such a global constant by the observation that higher Mercer eigenfunctions typically consist of high Fourier modes, and that these modes decay quickly under the dynamics. Therefore, high Mercer eigenfunctions should have a negligible share of the $p^\tau_{x}$ and the differences $p^\tau_{x_1}-p^\tau_{x_2}$.
For such $p^\tau_{x_1},p^\tau_{x_2}$, we thus have
\begin{equation}
\label{eq:H vs L2}
\|\mu(p^\tau_{x_1}) - \mu(p^\tau_{x_2}) \|_\H \geq \sqrt{\frac{\lambda_{i_\text{max}}}{c^*_{i_\text{max}}}}\| p^\tau_{x_1}-p^\tau_{x_2}\|_{L^2}.
\end{equation}

\paragraph{Contraction bound: dynamical requirements.}

Note that \eqref{eq:H vs L2} is only an intermediate step for deriving a contraction bound, as the relevant distance measure in Definition~\ref{def:good reaction coordinate} is the $L^2_{1/\rho}$-norm, for reasons detailed in Section~\ref{sec:transfer operators}, and~\eqref{eq:H vs L2} measures the density distance in the $L^2$-norm. Unfortunately, a naive estimation yields
\begin{equation}
\label{eq:naive estimation}
\|h \|_{L^2_{1/\rho}} \leq \left\|1/\sqrt{\rho}\right\|_\infty \|h\|_{L^2}.
\end{equation}
While due to ergodicity $1/\sqrt{\rho}$ is indeed defined on all of $\X$, it becomes large in regions of small invariant measure $\rho$, i.e., ``dynamically irrelevant'' regions. This would lead to a very large upper bound for the contraction. For general $h$, a more favorable estimate is indeed difficult to obtain. For us, however, $h = p^\tau_{x_1}-p^\tau_{x_2}$, and we can utilize that these ``dynamically irrelevant'' regions are almost never visited by the system.

To formalize this, we require one additional assumption that can be justified by the metastability of the system. One defining characteristic of metastable systems is the phenomenon that {essentially any} trajectory moves \emph{nearly instantaneously}\footnote{Here, ``nearly instantaneously'' is to be understood in the sense that there is an ``attraction time'' $\tau_a\ll \tau$ such that starting from essentially any initial condition the system will enter one of the metastable sets within time $\tau_a$ with overwhelming probability. Thus, choosing $\tau$ as an intermediate time that is larger than the non-metastable time scales of local fluctuations, is essential. However, as we discuss in Remark~\ref{rem:C}, it is also imperative not to choose $\tau$ too large.}  into one of the metastable sets before continuing. With $A_i,\ldots,A_d\subset \X$ denoting these sets, we can thus assume that the probability density $p^\tau_x(y)$ to move from $x$ to $y$ in time $\tau$ depends almost only on the probabilities to (instantaneously) move to the sets $A_i$ from $x$ (denoted by $c_i(x)$, thus $\sum_{i=1}^d c_i(x) \le 1$) and the probabilities to then move from $A_i$ to $y$ in time $\tau$ (denoted by $p^\tau_{A_i}(y)$), i.e.,
$$
p^\tau_x \approx \sum_{i=1}^d c_i(x) p^\tau_{A_i}.
$$
To be more precise, we require for all $x,y\in\X$ that
$$
\frac{\big\|p_x^\tau - \sum_{i=1}^d c_i(x) p^\tau_{A_i}\big\|_\infty}{ \|p_x^\tau\|_\infty} \approx 0,
$$
which is equivalent to
\begin{equation}
\label{eq:transition_density_decomposition}
p^\tau_x(y) = \frac{1}{1-\delta(x,y)}\sum_{i=1}^d c_i(x) p^\tau_{A_i}(y)
\end{equation}
for some positive function $\delta \colon \X\times\X\rightarrow \R$ with $\|\delta\|_\infty \leq \delta^* \ll 1$. The positivity of $\delta$ comes from the fact that there is a miniscule, but positive probability (density) to move from $x$ to $y$ without first equilibrating inside a metastable set, thus $p^\tau_x(y) > \sum_{i=1}^d c_i(x) p^\tau_{A_i}(y)$.

With this, we can bound the invariant density $\rho$ from below as follows:
\begin{align}
\rho(y) &= \int_\X \rho(x) p^\tau_x(y)\ts dx \nonumber\\
&= \int_\X \rho(x) \frac{1}{1-\delta(x,y)} \sum_{i=1}^d c_i(x) p^\tau_{A_i}(y)\ts dx \nonumber \\
&> \int_\X \rho(x)  \sum_{i=1}^d c_i(x) p^\tau_{A_i}(y)\ts dx \nonumber \\
&= \sum_{i=1}^d \underbrace{\Big( \int_\X \rho(x) c_i(x)\ts dx\Big)}_{=:b_i} p^\tau_{A_i}(y). \label{eq:rho_decomposition}
\end{align}
The $b_i$ can be seen as the equilibrium probability mass almost instantaneously attracted to~$A_i$. For every important metastable set this will not be too small.

As a first step, \eqref{eq:rho_decomposition} allows us to bound the $L^2_{1/\rho}$-norm of $p^\tau_x$ by their $L^1$-norm: 

\begin{lemma}
\label{lem:L1 estimation}
Let the assumption~\eqref{eq:transition_density_decomposition} hold and $b_i, i=1,\ldots,d,$ be defined as in \eqref{eq:rho_decomposition}. Then for any $x\in \X$ it holds that
\begin{equation}
\label{eq:L1 estimation}
\|p^\tau_x\|_{L^2_{1/\rho}}^2 < \frac{1}{(1-\delta^*)\min_i b_i} \|p_{x}^\tau\|_{L^1}
= \frac{1}{(1-\delta^*)\min_i b_i}
\end{equation}
\end{lemma}
\begin{proof}
See Appendix~\ref{sec:proofs}.
\end{proof}

This shows that indeed $p^\tau_x\in L^2_{1/\rho}$, as required by Definition~\ref{def:reducibility}. 
Further, Hölder's inequality gives
$$
\|f\|_{L^1} = \int_\X |f(x)| \cdot 1~dx \leq \|f\|_{L^2}\cdot |\X|^{1/2}.
$$
where $|\X|$ denotes the Lebesgue measure of the state space.
This now also allows us to bound the $L^2_{1/\rho}$-norm of the $p^\tau_x$ by their $L^2$-norm:

\begin{lemma}
\label{lem:L2 estimation}
Let the assumption~\eqref{eq:transition_density_decomposition} hold and $b_i, i=1,\ldots,d,$ be defined as in \eqref{eq:rho_decomposition}. Then for any $x_1,x_2\in \X$ it holds that
\begin{equation}
\label{eq:L2 estimation}
\|p^\tau_{x_1} - p^\tau_{x_2} \|_{L^2_{1/\rho}}^2 < \frac{|\X|^{1/2}}{(1-\delta^*)\min_i b_i} \|p_{x_1}^\tau - p_{x_2}^\tau \|_{L^2}.
\end{equation}
\end{lemma}
\begin{proof}
See Appendix~\ref{sec:proofs}.
\end{proof}

Of course, due to the squared norm on the left-hand side, this is not a Lipschitz bound. 
However, recall that our main motivation for deriving a bound for the contraction is to show that large distances in $L^2_{1/\rho}$ are not overly compressed under the embedding into $\H$, as illustrated in Figure~\ref{fig:Bad Embedding}.
We therefore abstain from deriving such a bound for very small distances in $L^2_{1/\rho}$ and only estimate the contraction of pairs of densities $p^\tau_{x_1}, p^\tau_{x_2}$ with
\begin{equation}
\label{eq:distance restriction}
\|p^\tau_{x_1}-p^\tau_{x_2}\|_{L^2_{1/\rho}} \geq C
\end{equation}
for some constant $C>0$. That $C$ is reasonably large is discussed in Remark~\ref{rem:C} below.
For such differences, we can then relate the $L^2$ to the $L^2_{1/\rho}$-norm, i.e.,
\begin{equation*}
\|p^\tau_{x_1} - p^\tau_{x_2} \|_{L^2_{1/\rho}} < \frac{ |\X|^{1/2}}{C(1-\delta^*)\min_i b_i} \|p_{x_1}^\tau - p_{x_2}^\tau \|_{L^2}.
\end{equation*}
Together with Lemma~\ref{lem:stretchingfactor} and assumption~\eqref{eq:assumption Mercer coefficients}, this gives
\begin{equation}
\label{eq:restricted contraction bound}
\|p^\tau_{x_1} - p^\tau_{x_2} \|_{L^2_{1/\rho}} < \frac{ |\X|^{1/2}}{C(1-\delta^*)\min_i b_i} \sqrt{\frac{\lambda_{i_\text{max}}}{c^*_{i_\text{max}}}} \ts \big\|\mu\big(p_{x_1}^\tau\big) - \mu\big(p_{x_2}^\tau\big) \big\|_{\H},
\end{equation}
which is our contraction bound.

\begin{remark} \label{rem:C}
For the distortion of the transition manifold under the embedding the essential property is that the global ``spanning structure'' of the manifold is well preserved. In other words, the embedded $p_x^{\tau}, p_y^{\tau}$ should be well-separated for $x,y \in \X$ from different metastable sets~$A_i$. Since the embedding is continuous, the transition paths connecting them will be preserved as well.

As $p^{\tau}_{A_i} \to \rho$ as $\tau\to \infty$ for every $i$, it is important that $\tau$ is not too large, such that the transition manifold is a meaningful object. In a other words, $\tau$ should be such that $p^{\tau}_{A_i}$ and $p^{\tau}_{A_j}$ are sufficiently distinct for $i\neq j$. Thus, we \emph{require} that $\smash{\| p^{\tau}_{A_i} - p^{\tau}_{A_j} \|_{L^2_{1/\rho}}}$ is sufficiently large for $i\neq j$, hence $C$ can be chosen as a constant such that $1/C$ is reasonably small.

\end{remark}

\begin{remark}
The bounds~\eqref{eq:expansion bound} and \eqref{eq:restricted contraction bound} guarantee the well-posedness of the overall embedding and pa\-ra\-me\-tri\-za\-tion problem as the relevant expansion and contraction of the transition manifold cannot become arbitrarily large. We will support this statement with numerical evidence (see Section~\ref{sec:Horseshoe potential}) showing that the distortion is, in practice, indeed small.
It should be noted, however, that we do not expect the analytically derived bounds to perform well as quantitative error estimates, as many of the estimates that led to them are rather rough.
\end{remark}

}

\section{Illustrative examples and applications}
\label{sec:examples}

\subsection{Reaction coordinate of the M\"uller--Brown potential}

As a first illustrating example, we compute the reaction coordinate of the two-dimensional M\"uller--Brown potential \cite{Mueller1980} via the new kernel-based Algorithm~\ref{algo:KernelRC}. The potential energy surface (see Figure~\ref{fig:MB_pot} (a)) possesses three local minima, where the two bottom minima are separated only by a rather shallow energy barrier. Correspondingly, the system's characteristic long-term behavior is determined by the rare transitions between the minima. These transitions happen predominantly along the potential's \emph{minimum energy pathway} (MEP), which is shown as white dashed line and was computed using the zero temperature string method \cite{E2002,E2007}.

For two sets $A,B\subset \X$ and a starting point $x\in\X$, the \emph{committor function} $q_{AB}(x)$ is defined as the probability that the process hits set $A$ before hitting set $B$, provided it started in $x$ at time zero. For a precise definition see~\cite{SS13}. For the M\"uller--Brown potential, the committor function associated with the top left and bottom right energy minima, shown in Figure~\ref{fig:MB_pot}~(b), can be considered an optimal reaction coordinate~\cite{ElEtAl17}. Therefore, we use the (qualitative) comparison with the committor function as a benchmark for our reaction coordinate.
Note that the computation of the committor function requires global knowledge of the metastable sets and is often not a practical option for the identification of reaction coordinates.

\begin{figure}
    \centering
    \begin{tabular}{ccc}
    (a)
    &
    (b)
    &
    (c)
    \\
    \includegraphics[scale=0.85]{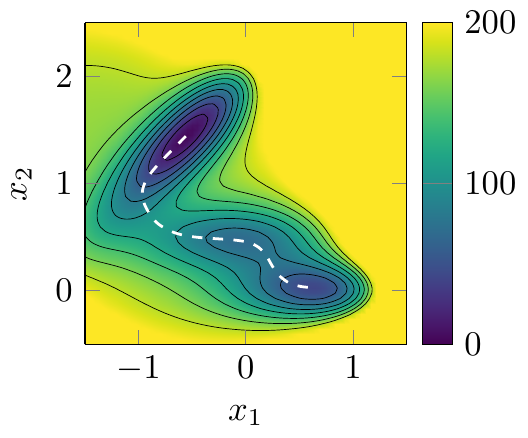}
    &
    \includegraphics[scale=0.85]{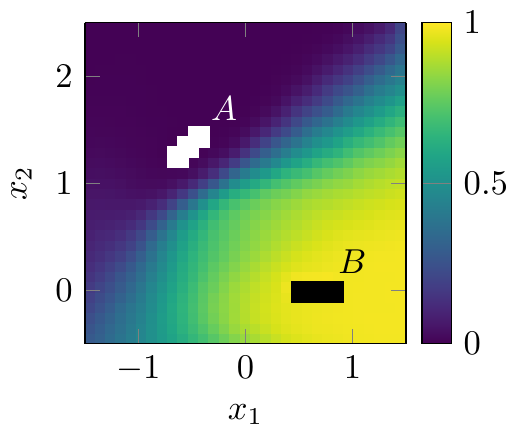}
    &
    \includegraphics[scale=0.85]{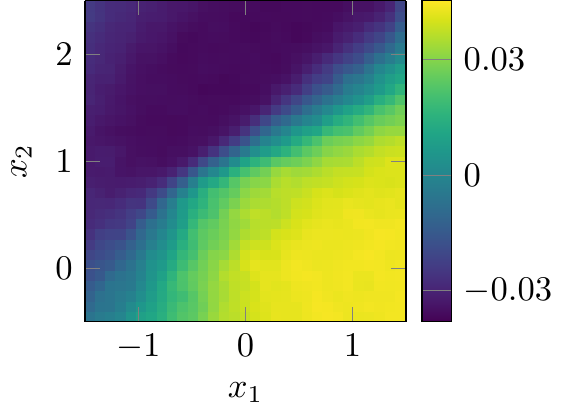}
    \end{tabular}
    \caption{(a) The M\"uller--Brown potential energy function with its three characteristic local minima and the connecting MEP (white line). (b) The committor function $q_{AB}$ associated with the areas around the top left ($A$) and bottom right $(B)$ energy minimum. (c) Reaction coordinate $\xi$ of the Müller--Brown potential, computed by Algorithm~\ref{algo:KernelRC}.}
    \label{fig:MB_pot}
\end{figure}

The governing dynamics is given by an overdamped Langevin equation~\eqref{eq:Langevin dynamics}, which we solve numerically using the Euler--Maruyama scheme. At inverse temperature $\beta=0.05$, the lag time $\tau=0.03$ falls between the slow and fast time scales and is thus defined to be the intermediate lag time. The test points $\{x_1,\ldots,x_N\}$ required by Algorithm~\ref{algo:KernelRC} are given by a regular $32\times 32$ grid discretization of the domain $[-1.5, 1.5]\times[-0.5, 2.5]$. For the embedding, the Gaussian kernel
\begin{equation}
\label{eq:Gaussian kernel}
k(x,x') = \exp\left( -\frac{\|x-x'\|_2^2}{\sigma} \right)
\end{equation}
with bandwidth $\sigma=0.1$ is used and for the manifold learning task in Algorithm~\ref{algo:KernelRC}, the diffusion maps algorithm with bandwidth parameter $0.1$.
The reaction coordinate $\xi$ for the test points is shown in Figure~\ref{fig:MB_pot}~(c). We observe remarkable resemblance to the committor function.

\hl{
\begin{remark}
The kernel evaluations used for the RKHS embedding of densities and the kernel evaluations used in the diffusion maps algorithm should not be mixed up as they serve different purposes.
The former is used to embed the state space densities into $\mathbb{H}$, while the latter is used to approximate the Laplace--Beltrami operator on the manifold in $\mathbb{H}$ that is to learn (this is the principle on which the diffusion maps algorithm is based).
Even though the Gaussian kernel is a popular choice due to its favorable characteristics, one has great freedom in choosing a kernel for the RKHS-embedding, whereas in the classical diffusion maps algorithm, predominantly the Gaussian kernel is used, so the repeated use of the Gaussian kernel does not constitute a connection.
Moreover, the fact that in this example identical bandwidth parameters were used was a mere coincidence.
We do not see a way to unify these kernel evaluations, neither on a conceptual nor algorithmic level.
\end{remark}
}

\subsection{Distortion under the Whitney and kernel embeddings}
\label{sec:Horseshoe potential}

We now demonstrate the distortion of the transition manifold under the embedding via \hl{the conventional} Algorithm~\ref{algo:WhitneyRC} (Whitney embedding) and \hl{our new} Algorithm~\ref{algo:KernelRC} (kernel-based embedding). To this end, we consider the two-dimensional potential depicted in Figure~\ref{fig:Horseshoe potential}~(a). Good reaction coordinates of a diffusion process in this potential should parametrize the ``horseshoe-like'' MEP shown as a white dashed line. Such a reaction coordinate is depicted in Figure~\ref{fig:Horseshoe potential}~(b).

\begin{figure}
    \centering
    \begin{tabular}{cc}
    (a) & (b) \\
    \includegraphics[scale=1]{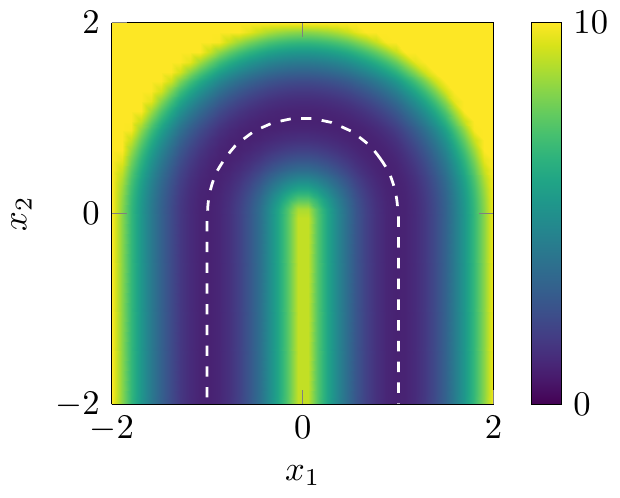}
    &
    \includegraphics[scale=1]{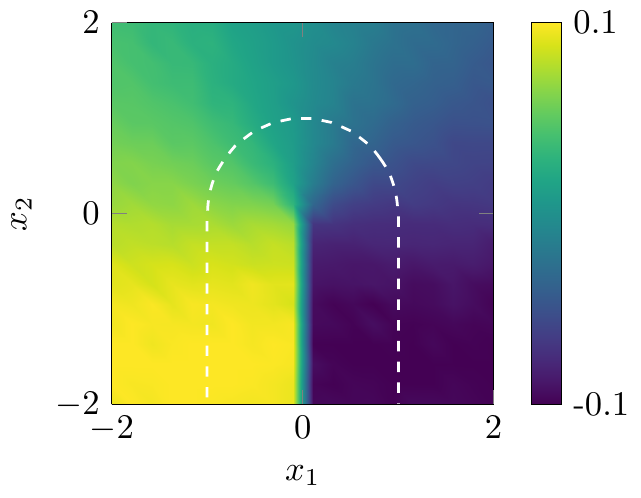}
    \end{tabular}
    \caption{Horseshoe potential. (a) Potential energy function. (b) Reaction coordinate computed with the kernel algorithm. The white dashed line represents the MEP.}
    \label{fig:Horseshoe potential}
\end{figure}

As test points $x_i$, $N=200$ uniformly distributed random points in the region $\X=[-2,2]^2$ are drawn. Per point, $M=100$ short trajectories of length $\tau=2$ are computed to sample the~$p^\tau_{x_i}$.

\subsubsection*{Whitney embedding}

For the Whitney embedding, the expected manifold dimension $r=1$ is assumed to be known in advance. To demonstrate the different effects of ``good'' and ``bad'' embedding functions, two $2r+1$-dimensional linear observables $\eta:\R^2\rightarrow\mathbb{R}^3$ were chosen:
$$
\eta_g \colon x\mapsto A_g\ts x,\qquad \eta_b \colon x\mapsto A_b\ts x,\qquad A_g,A_b\in\R^{3\times 2},
$$
and the corresponding embedding functions $\mathcal{F}_g, \mathcal{F}_b$ constructed via~\eqref{eq:integral_embedding}.
The coefficients of $A_g$ of the 
observable function $\eta_g$ were chosen randomly via the Matlab command
$$
\verb!rng(1); A_g=rand(3,2)-0.5!,
$$
\hl{
which resulted in the matrix
$$
A_g \approx \begin{pmatrix}
-0.08 & -0.20 \\
0.22 & -0.35\\
-0.49 & -0.41
\end{pmatrix}.
$$
}
\hl{
Under the embedding $\mathcal{F}_g$, all parts of the transition manifold can indeed be distinguished very well, see Figure~\ref{fig:Horseshoe Whitney}~(a). This is due to the fact that two distinct points of the transition pathway of the potential, for example on the two opposite ``branches'', are never mapped to the same point under $\eta_g$.}

On the other hand, the matrix $A_b$ of the ``bad'' observable function $\eta_b$ was \hl{intentionally} constructed to consist of three row vectors that are pairwise almost linearly dependent, \hl{and that essentially ignore the $x_1$-component of state space points:}
$$
A_b = \begin{pmatrix}
0 & 1 \\
\varepsilon & 1+\varepsilon\\
-\varepsilon & 1-\varepsilon
\end{pmatrix}
\qquad \text{with } \varepsilon=0.05.
$$
\hl{
This way, points on the two opposite branches of the transition pathway but with the same $x_2$-coordinate are mapped to almost the same point in $\mathbb{R}^3$. The result is an embedding of the transition manifold in which the two branches can hardly be distinguished, see Figure~\ref{fig:Horseshoe Whitney}~(b). This makes the numerical identification of the manifold structure extremely difficult.}

\hl{
Note that the judgment of quality of the embedding function has to be performed manually after the embedding, as it is impossible to reliably choose good embedding functions without detailed a priori knowledge of the global structure of the transition manifold or transition pathway. While in our experience, randomly chosen coefficients typically result in ``good-enough'' embedding functions, this uncertainty in the numerical algorithm should be seen as one of the main reasons to use the more consistent kernel embeddings instead.
}
\begin{figure}
\centering
\begin{tabular}{cc}
(a) & (b) \\
\includegraphics[scale=1]{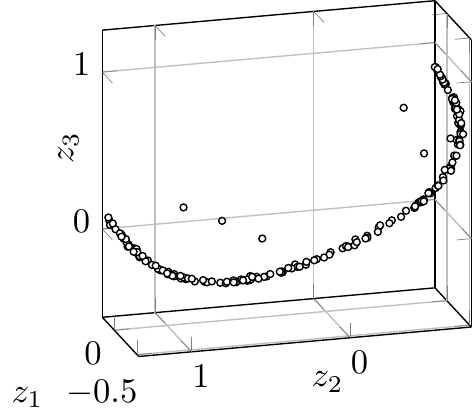}
&
\includegraphics[scale=1]{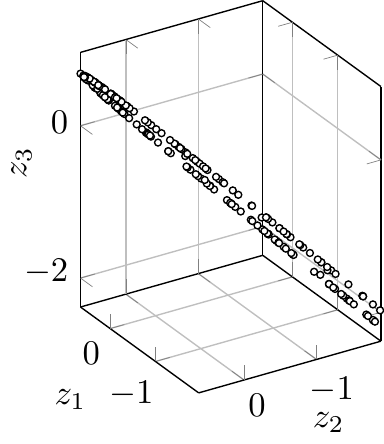}
\end{tabular}
\caption{Whitney embeddings of the test points for different observable functions. (a)~Pairwise strongly linearly independent coefficient vectors, i.e., ``good'' observables. (b) Almost pairwise linearly dependent coefficient vectors, i.e., ``bad'' observables.}
\label{fig:Horseshoe Whitney}
\end{figure}

\subsubsection*{Kernel embedding}

For the kernel embedding, we again utilize the Gaussian kernel~\eqref{eq:Gaussian kernel} with bandwidth~$\sigma=10^{-3}$. 
The result of the kernel embedding of the test points is an approximation of the kernel distance matrix 
$$
D_{ij} = \big\|\mu(p^t_{x_i}) - \mu(p^t_{x_j})\big\|_\H
$$ 
(see Algorithm~\ref{algo:KernelRC}), which cannot be visualized directly. We thus apply the Multidimensionl scaling (MDS) algorithm to $D$, in order to visualize the level of similarity between the embedded densities.

Given a distance matrix $D$, MDS generates points $z_i\in\R^k$ in a Euclidean space of a chosen dimension $k\in\N$ such that the pairwise distance between the $z_i$ corresponds to the distances in $D$. For an overview of different MDS methods, see for example~\cite{Young2013}. We here use the implementation of classical MDS given by the \verb!cmdscale! method in Matlab.

The MDS representation of the kernel distance matrix for $k=2$ is shown in Figure~\ref{fig:MDS representations}~(a). The curved structure of the MEP is immediately visible.
Moreover, it is also possible to visualize the corresponding $L^2_{1/\rho}$ and $L^2$ distance matrices via MDS, i.e., the matrices
$$
\big(D_{L^2_{1/\rho}}\big)_{ij}:=\big\| p^t_{x_i} - p^t_{x_j} \big\|_{L^2_{1/\rho}} \quad \text{and}\quad \big(D_{L^2}\big)_{ij}:= \big\| p^t_{x_i} - p^t_{x_j} \big\|_{L^2}.
$$
The results are shown in Figure~\ref{fig:MDS representations}~(b)\ts\&\ts(c).

\begin{figure}
\centering
\begin{tabular}{ccc}
(a) & (b) & (c) \\
\includegraphics[scale=0.9]{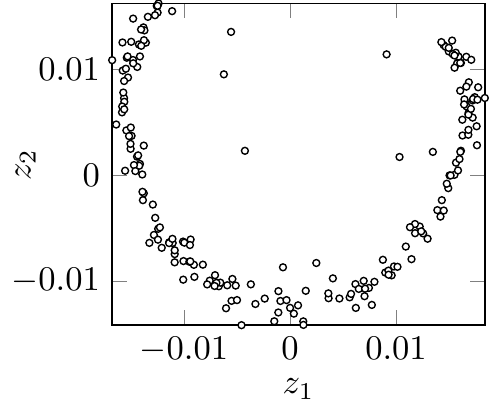}
&
\includegraphics[scale=0.9]{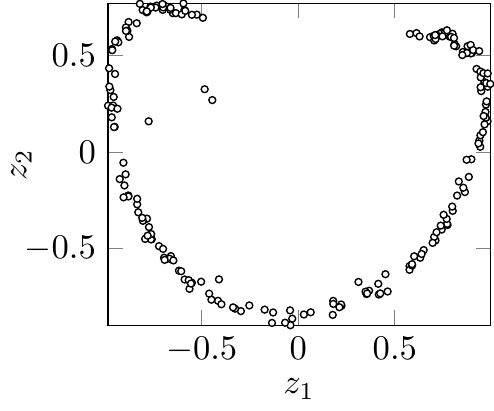}
&
\includegraphics[scale=0.9]{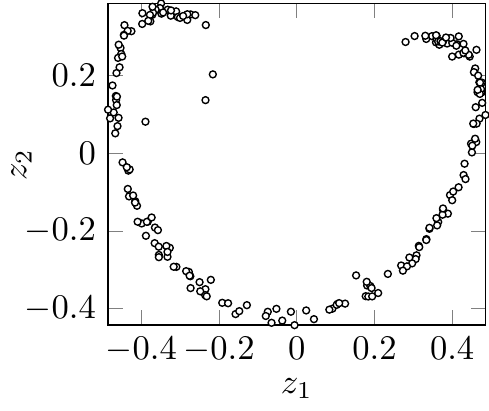}
\end{tabular}
\caption{MDS representations of different distance matrices between the transition densities. (a) Kernel distance matrix between embedded transition densities $\mu(p^t_x)$. The point cloud is a representations of the fuzzy transition manifold embedded into $\H$. (b) $L^2_{1/\rho}$ distance matrix between transition densities $p^t_x$. (c) $L^2$ distance matrix between transition densities~$p^t_x$. }
\label{fig:MDS representations}
\end{figure}

The MDS representation of $D$ is structurally very similar to $D_{L^2_{1/\rho}}$ and $D_{L^2}$. This suggests that the $L^2_{1/\rho}$ and $L^2$ distances are preserved very well under $\mu$, up to a constant factor.
To confirm this, we now compute the \hl{empirical} maximum distortion of the $L^2_{1/\rho}$ metric based on the given test points, i.e., $\mathscr{D}_N(\mu):= \mathscr{C}_N(\mu) \ts \mathscr{E}_N(\mu)$ where
$$
\mathscr{C}_N(\mu) := \max_{\substack{i,j=1,\ldots,N \\ i\neq j}} \frac{\|p^\tau_{x_i}-p^\tau_{x_j}\|_{L^2_{1/\rho}}}{\|\mu(p^\tau_{x_i})-\mu(p^\tau_{x_j})\|_\H}, \qquad \mathscr{E}_N(\mu) := \max_{\substack{i,j=1,\ldots,N \\ i\neq j}} \frac{\|\mu(p^\tau_{x_i})-\mu(p^\tau_{x_j})\|_\H}{\|p^\tau_{x_i}-p^\tau_{x_j}\|_{L^2_{1/\rho}}}.
$$
For large enough $N$, we expect $\mathscr{D}_N(\mu)$ to be a good estimator for the true distortion $\mathscr{D}(\mu)$. 

The blue graph in Figure~\ref{fig:Horseshoe distortion} shows the dependence of \hl{the empirical} distortion on the kernel parameter $\sigma$. \hl{Here the minimum is $\mathscr{D}_N(\mu)= 36.7$ at $\sigma\approx 10^{-3}$. Interpreting $\mathscr{D}_N(\mu)$ as the condition number of the kernel-based embedding problem, the problem can be described as reasonably well-conditioned.}

\hl{
Analogously, we can define the \hl{empirical} maximum distortion of the Whitney embedding as $\mathscr{D}_N(\mathcal{F}):= \mathscr{C}_N(\mathcal{F}) \ts \mathscr{E}_N(\mathcal{F})$, where
$$
\mathscr{C}_N(\mathcal{F}) := \max_{\substack{i,j=1,\ldots,N \\ i\neq j}} \frac{\|p^\tau_{x_i}-p^\tau_{x_j}\|_{L^2_{1/\rho}}}{\|\mathcal{F}(p^\tau_{x_i})-\mathcal{F}(p^\tau_{x_j})\|_{\R^3}}, \qquad \mathscr{E}_N(\mathcal{F}) := \max_{\substack{i,j=1,\ldots,N \\ i\neq j}} \frac{\|\mathcal{F}(p^\tau_{x_i})-\mathcal{F}(p^\tau_{x_j})\|_{\R^3}}{\|p^\tau_{x_i}-p^\tau_{x_j}\|_{L^2_{1/\rho}}}.
$$
For a given embedding $\mathcal{F}$, this distortion can again be computed numerically. For the ``good'' embedding $\mathcal{F}_g$, we obtain $\mathscr{D}_N(\mathcal{F}_g)\approx 6\cdot 10^2$, while for the ``bad'' embedding $\mathcal{F}_b$, we obtain $\mathscr{D}_N(\mathcal{F}_b)\approx 7\cdot 10^3$. The kernel embedding is therefore much better conditioned than both Whitney embeddings.}

\begin{remark}
Analogously, we can also define and compute the maximum distortion $\mathscr{D}_N(\mu)$ of the $L^2$-metric (red graph in Figure~\ref{fig:Horseshoe distortion}). Here, for $\sigma\approx 10^{-1}$ we obtain $\mathscr{D}_N(\mu)=2.7$, i.e.,  the embedding becomes nearly isometric. 
This is not surprising as it has been shown in~\cite{Sri10} that for radial kernels $k_\sigma(x,y) = \sigma^{-d} g(\sigma^{-1}\|x-y\|)$ where $g$ is bounded, continuous, and positive definite, it holds that
$$
\lim_{\sigma\rightarrow 0} \|\mu_{k_\sigma}(p) - \mu_{k_\sigma}(q)\|_\H = \|p-q\|_{L^2}.
$$
The Gaussian kernel belongs to this class of kernels.
We thus expect that by increasing the sample number $M$ of the transition densities and further decreasing $\sigma$, the distortion can be reduced further.
However, recall that for our application, only the distortion of the $L^2_{1/\rho}$ distance is relevant.
\end{remark}

\begin{figure}
\centering
\includegraphics[scale=1]{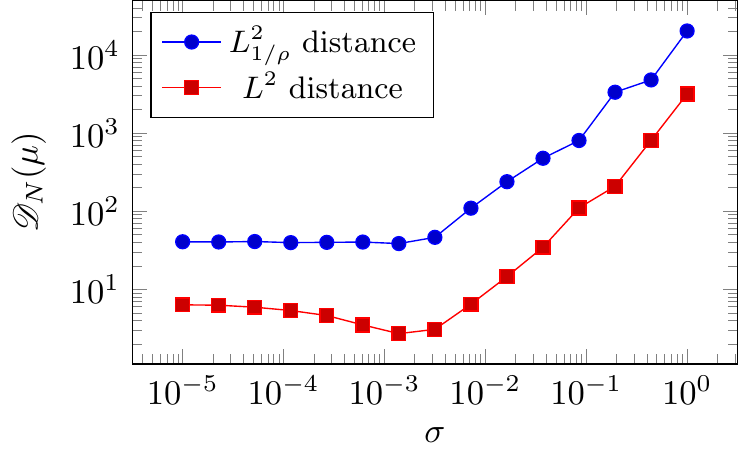}
\caption{Maximum distortion $\mathscr{D}(\mu)$ of the $L^2_{1/\rho}$ and $L^2$ distance under the kernel embedding $\mu$ for the Gaussian kernel depending on the kernel bandwidth $\sigma$.}
\label{fig:Horseshoe distortion}
\end{figure}

\subsection{Alanine dipeptide}

We now demonstrate the applicability of Algorithm~\ref{algo:KernelRC} to realistic, high-dimensional molecular systems by computing reaction coordinates of the Alanine dipeptide. The peptide, depicted in Figure~\ref{fig:aladipep_structure}~(a), consists of 22 atoms, the state space $\X$ thus has the dimension $n=66$.

\begin{figure}[h]
\centering
\begin{minipage}{.3\textwidth}
\centering
(a)\\
\includegraphics[width=\textwidth]{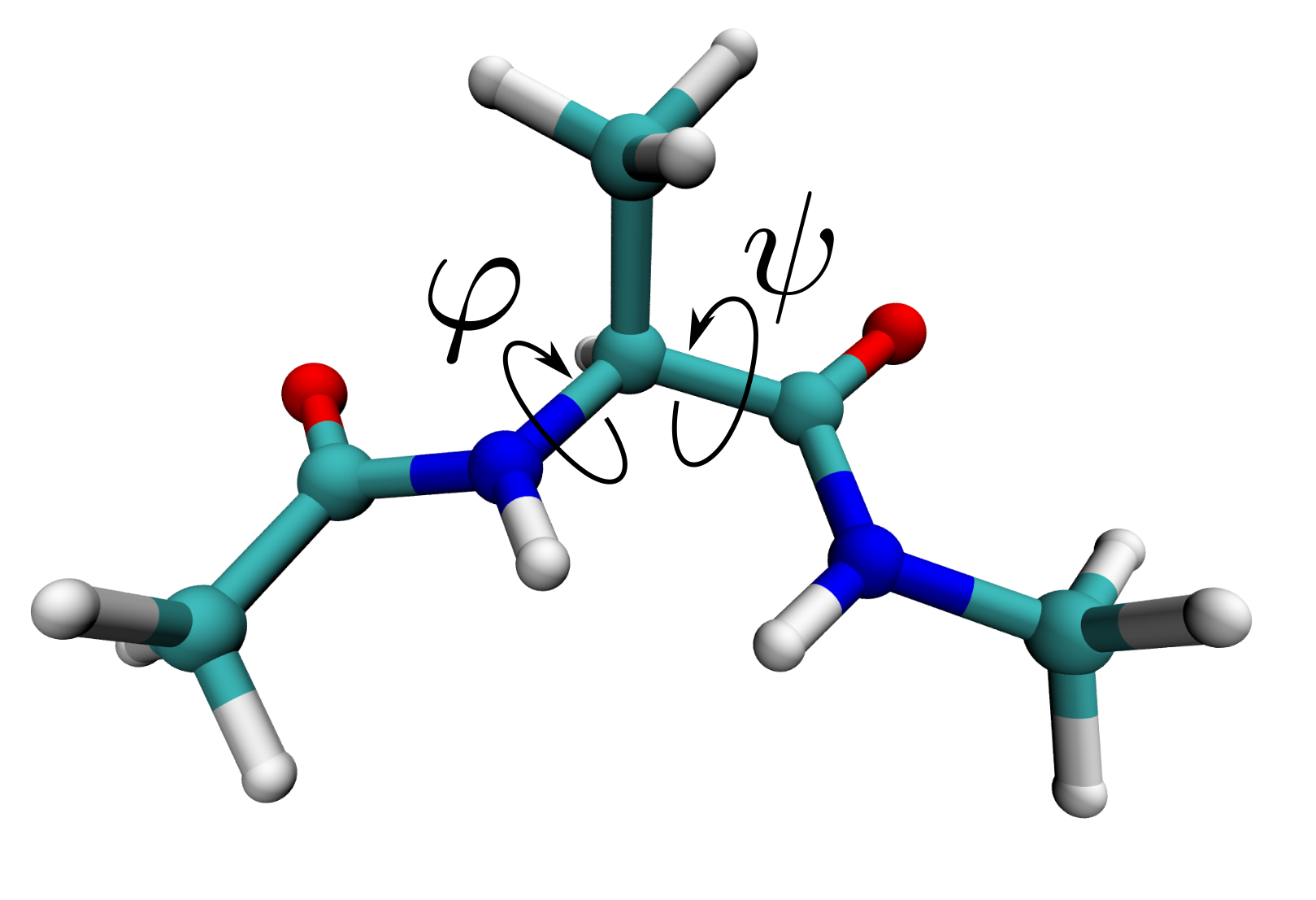}
\end{minipage}
\begin{minipage}{.3\textwidth}
\centering
(b)\\
\includegraphics[width=\textwidth]{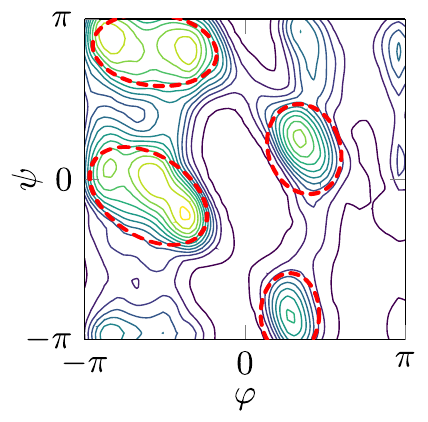}
\end{minipage}
\caption{The Alanine dipeptide. (a) Three-dimensional structure with the two essential dihedral angles $(\varphi,\psi)$ highlighted. (b) The Ramachandran plot of $(\varphi,\psi)$ reveals four local energy minima, i.e., metastable sets.}
\label{fig:aladipep_structure}
\end{figure}

It is well-known that the essential long-term behavior of this system is governed by the metastable transitions between four local minima of the potential energy surface (PES)~\cite{Chekmarev2004,Smith1999}. These minima are clearly visible when projecting the PES onto two specific backbone dihedral angles $(\varphi,\psi)$ that we call \emph{essential} from now on (see Figure~\ref{fig:aladipep_structure}~(b)). The transition between the metastable states happens along minimal energy pathways that we aim to reveal with our reaction coordinate.
Note however that no information about the existence of the two essential dihedral angles was used in our experiments, and we perform all of the analysis on the full 66-dimensional data.


The simulations were performed using the Gromacs molecular dynamics software~\cite{Berendsen1995}. We consider the molecule in explicit aqueous solution at temperature $400\,$K (the water molecules are discarded prior to further analysis). To generate the test points $x_i$, $N=1000$ snapshots from a long, equilibrated trajectory were subsampled. This guarantees that the $x_i$ cover the dynamically relevant regions of $\X$, i.e., the metastable sets and transition pathways. \hl{The values of the dihedral angles $\varphi$ and $\psi$ of the test points are shown in Figure~\ref{fig:aladipep diheral comparison} (the $x$- and $y$-coordinates of the points). We see that the metastable sets and transition pathways from Figure~\ref{fig:aladipep_structure}~(b) are adequately covered. Note however that the projection onto the $(\varphi,\psi)$-space here serves only illustrative purposes; we continue to work with the test points in the full 66-dimensional space.}

The intermediate lag time $\tau=20\,\mathrm{ps}$ falls between the slow and fast time scales. For strategies to estimate $\tau$ prior to simulation, see~\cite{BBS18}. For each test point $x_i$, $M=100$ simulations of length $\tau$ were performed, which took $40\,\mathrm{h}$ on a 96 core cluster. The resulting point clouds $\{y_i^{(l)}$, $l=1,\ldots,M\}$ are samplings of the densities $p^\tau_{x_i}$.

To compute the kernel distance matrix $D$ from the simulation data, the Gaussian kernel~\eqref{eq:Gaussian kernel} with bandwidth $\sigma=0.1$ was chosen. 
For the plug-in manifold learning algorithm that is applied to $D$, the diffusion maps algorithm with bandwidth parameter $0.01$ was used.
The analysis of the simulation data was performed in Matlab and took $4$ minutes on a 4 core laptop.


Figure~\ref{fig:aladipep analysis}~(a) shows the leading spectrum of the diffusion map matrix that was computed based on $D$. The first diffusion map eigenvalue is always one, and the associated eigenvector carries no structural information. Therefore, a spectral gap after the third sub-dominant eigenvalue indicates that the underlying transition manifold is intrinsically three-dimensional. The associated three subdominant eigenvectors now are the final reaction coordinate. For each of the 1000 test points, the values of the three eigenvectors are shown in Figure~\ref{fig:aladipep analysis}~(b). This can be seen as the embedding of the test points into the reaction coordinate space. Here we observe four clusters of points, and three connecting paths. 
\hl{In Figure~\ref{fig:aladipep diheral comparison}, the values of the dihedral angles $\varphi$ and $\psi$ at the test points are compared to the values of the three components of the computed reaction coordinates, shown in color. We see that areas of almost constant color correspond to the four metastable sets from Figure~\ref{fig:aladipep_structure}. Thus, the computed reaction coordinate is able to identify the metastable sets and resolve transitions between them.}

\begin{figure}
\centering
\begin{minipage}{.34\textwidth}
\centering
(a)\\
\includegraphics[scale=1]{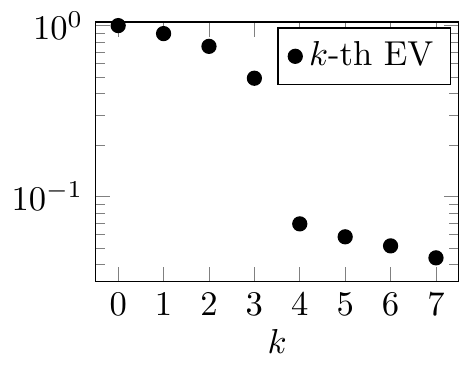}
\end{minipage}
\begin{minipage}{.5\textwidth}
\centering
(b)\\
\includegraphics[scale=1]{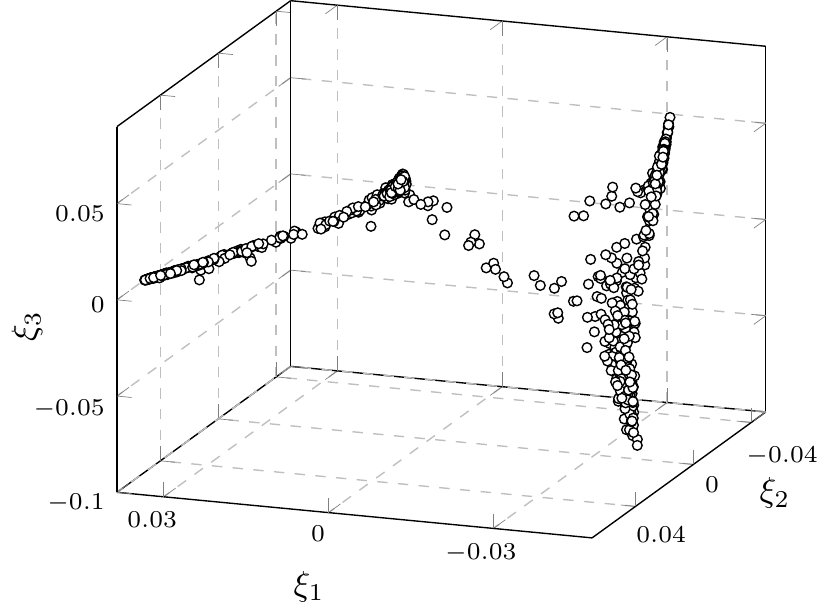}
\end{minipage}
\caption{Analysis of the kernel distance matrix $D$. (a) Eigenvalues of the diffusion map matrix. The existence of three eigenvalues close to 1 (not counting the eigenvalue 1 itself) indicates a three-dimensional reaction coordinate. (b) Test points in the space of the three sub-dominant diffusion map eigenvectors, i.e., the final three-dimensional reaction coordinate.
}
\label{fig:aladipep analysis}
\end{figure}

\begin{figure}
\centering
\begin{tabular}{ccc}
\includegraphics[scale=0.9]{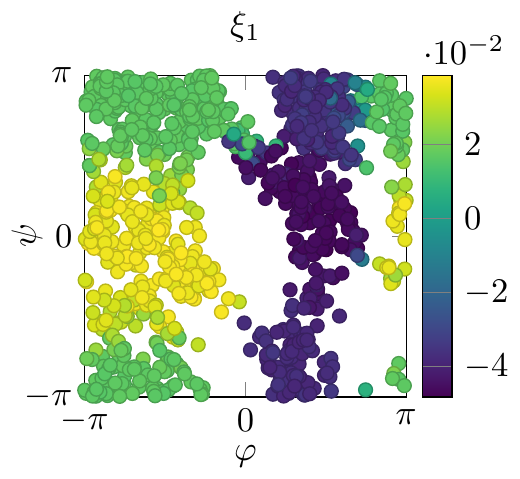}
&
\includegraphics[scale=0.9]{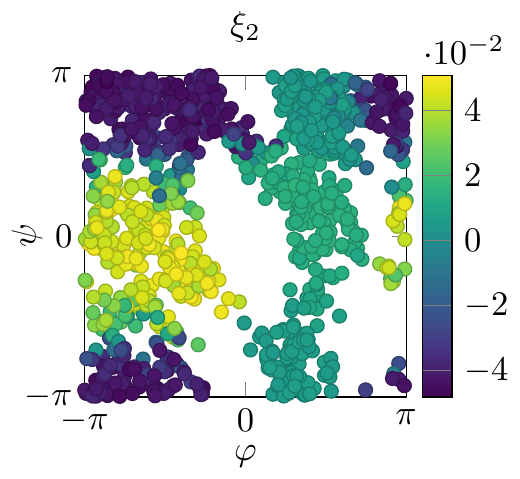}
&
\includegraphics[scale=0.9]{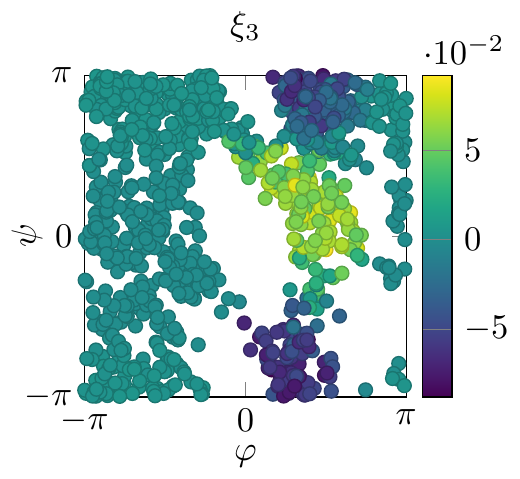}
\end{tabular}
\caption{The dihedral angles in the test points, colored by the three components of the reaction coordinate $\xi$. The coordinate \hl{$\xi_1$ primarily describes transitions in the angle $\varphi$, whereas $\xi_2$ and $\xi_3$ describe transitions in the angle $\psi$ for low and high values of $\psi$, respectively.}}
\label{fig:aladipep diheral comparison}
\end{figure}

\section{Conclusion and future work}
\label{sec:Conclusion}

In this work, we have analyzed the embedding of manifolds that lie in certain function spaces into reproducing kernel Hilbert spaces. Moreover, we have proposed efficient numerical algorithms for learning parametrizations of these embedded manifolds from data.
The question is motivated by the recent insight that parametrizations of the so-called transition manifold, a manifold consisting of the transition density functions of a stochastic system, are strongly linked to reduced coordinates for that system. The method can thus be used for coarse graining a given system.

Compared to previous approaches based on random embeddings into a Euclidean space, the new kernel-based approach eliminates the need to know the transition manifold dimension a priori. Furthermore, if a universal kernel is used, the topological structure of the transition manifold is guaranteed to be preserved under the embedding. We have derived bounds for the geometric distortion of the transition manifold under the RKHS embedding, which can be interpreted as the condition of the overall coarse graining procedure. Correspondingly, the numerical algorithm was demonstrated to be very robust, especially when compared to random embeddings, and, in realistic applications, we obtained very favorable results regarding algorithmic distortion bounds.

There are several new avenues to use the broader theory of kernel embeddings to characterize the kernel embedding of transition manifolds. First, we plan to improve the theoretic distortion bounds derived in Section~\ref{sec:kernel mean embedding} by considering different established interpretations of the metric defined by $d(p,q)=\|\mu(p) - \mu(q)\|_\H$. For an overview, see~\cite{Sri10}.

Recently, the spectral theory of transfer operators was extended to reproducing kernel Hilbert spaces in~\cite{KSM17}. The usefulness of this new theory for the data-driven conformation analysis of molecular systems was demonstrated in~\cite{KBSS18}. As the transition manifold can be defined via the transfer operator\footnote{The fuzzy transition manifold is the image of all Dirac densities under the transfer operator, i.e., $\widetilde{\mathbb{M}} = \{ \mathcal{T}^t \delta_x \mid x\in\X\}$.}, it seems natural to attempt to relate the embedded transition manifold to the kernel transfer operators and corresponding embedded transfer operators defined in~\cite{KSM17}.

Finally, as illustrated in \cite{BH10, BouHa17, BouHaJCD}, RKHSs can act as \emph{linearizing spaces} in the sense that performing linear analysis in the RKHS can capture strong nonlinearities in the original system.
A typical example is the problem of linear separability in data classification: A data set which is not linearly separable might be easily separated when mapped into a nonlinear feature space. In our current context, this means that efficient linear manifold learning methods might be suitable to parametrize the embedded manifold, if the kernel is chosen appropriately. We will investigate whether a corresponding theory can be developed.

\section*{Acknowledgements}

The authors would like to thank the anonymous reviewers for constructive comments and suggestions that helped to improve the paper.

This research has been partially funded by Deutsche Forschungsgemeinschaft (DFG) through grant CRC 1114, projects A01,~B03, and~B06.

\bibliographystyle{abbrv}
\bibliography{KernelTM}

\appendix
\section{Proof of the distortion bounds}
\label{sec:proofs}

\begin{proof}[Proof of Proposition~\ref{prop:small distance}]
First, note that $L^2_{1/\rho}\subset L^2 $, as for $p\in L^2_{1/\rho}$
\begin{equation}
\label{eq:sqrtrho}
\|p\|_{L^2} = \|\sqrt{\rho}\ts p\|_{L^2_{1/\rho}} \leq \|\sqrt{\rho}\|_\infty \ts \|p\|_{L^2_{1/\rho}}.
\end{equation}

Let $(\lambda_i,\varphi_i)$ be the eigenpairs of the integral operator $\mathcal{T}_k$, ordered in decreasing order of $\lambda_i$. For arbitrary $p\in L^2_{1/\rho}$ consider the decomposition into the basis $\big\{\varphi_i\big\}_{i\in\mathbb{N}}$ of $L^2$:
$$
p = \sum_{i=0}^\infty \tilde{p}_i \ts \varphi_i.
$$
Select $i_\text{max}\in\N$ such that there exists an index $i\leq i_\text{max}$ with $\tilde{h}_i\neq 0$, and such that $\lambda_i<\big(\varepsilon/\|\sqrt{\rho}\|_\infty\big)^2$ for all $i\geq i_\text{max}$, and define
$$
q = \sum_{i=0}^{i_\text{max}-1}\tilde{p}_i \varphi_i.
$$
Then,
 $$
\|p-q\|_{L^2}^2 = \Big\| \sum_{i=i_\text{max}}^\infty \tilde{p}_i \varphi_i\Big\|_{L^2}^2 = \sum_{i=i_\text{max}}^\infty \tilde{p}_i^2.
$$
Further, using that $\big\{\sqrt{\lambda_i} \ts \varphi_i\big\}_{i\in\mathbb{N}}$ forms an orthonormal basis of $\H$, and that $\mu$ is a linear operator, we get
$$
\|\mu(p)-\mu(q)\|_\H^2 = \Big\|\sum_{i=i_\text{max}}^\infty \tilde{p}_i \lambda_i \varphi_i \Big\|_\H^2 = \sum_{i=i_\text{max}}^\infty \lambda_i \tilde{p}_i^2 \leq \lambda_{i_\text{max}}\sum_{i=i_\text{max}}^\infty \tilde{p}_i^2 .
$$
Thus we get
$$
\frac{\|p-q\|_{L^2}}{\|\mu(p)-\mu(q)\|_\H} \geq 1/{\sqrt{\lambda_{i_\text{max}}}} > \|\sqrt{\rho}\|_\infty/\varepsilon,
$$
and with \eqref{eq:sqrtrho} finally
\begin{equation*}
    \frac{\|p-q\|_{L^2_{1/\rho}}}{\|\mu(p)-\mu(q)\|_\H} > \frac{1}{\varepsilon}. \qedhere
\end{equation*}
\end{proof}

\begin{proof}[Proof of Lemma~\ref{lem:stretchingfactor}]
As $\big\{\sqrt{\lambda_i} \ts \varphi_i\big\}_{i\in\mathbb{N}_0}$ forms an orthonormal basis of $\H$, we obtain
$$
\|\mu(h)\|_\H^2 = \Big\|\sum_{i=0}^\infty \tilde{h}_i \lambda_i \varphi_i \Big\|_\H^2 = \sum_{i=0}^\infty \lambda_i \tilde{h}_i^2 \geq \sum_{i=0}^{i_\text{max}} \lambda_i \tilde{h}_i^2.
$$
Further, $\big\{\varphi_i\big\}_{i\in\mathbb{N}_0}$ forms an orthonormal basis of $L_2$, and so
$$
\|h\|_2^2 = \Big\| \sum_{i=0}^\infty \tilde{h}_i \varphi_i\Big\|_2^2 = \sum_{i=0}^\infty \tilde{h}_i^2 = c(h,i_\text{max}) \cdot \sum_{i=0}^{i_\text{max}} \tilde{h}_i^2.
$$
Thus,
\begin{equation*}
\frac{\|\mu(h)\|_\H}{\|h\|_2} \geq \left(\frac{\sum_{i=0}^{i_\text{max}}\lambda_i\tilde{h}_i^2}{c(h,i_\text{max})\cdot\sum_{i=0}^{i_\text{max}}\tilde{h}_i^2}\right)^{1/2} \geq \sqrt{\frac{\lambda_{i_\text{max}}}{c(h,i_\text{max})}}. \qedhere
\end{equation*}
\end{proof}

\hl{

\begin{proof}[Proof of Lemma~\ref{lem:L1 estimation}]
With assumption~\eqref{eq:transition_density_decomposition} and~\eqref{eq:rho_decomposition}, we can write the left-hand side of~\eqref{eq:L1 estimation} as
\begin{align*}
\|p^\tau_{x} \|_{L^2_{1/\rho}}^2 &= \int_\X p^\tau_{x}(y)^2 \frac{1}{\rho(y)}\ts dy \\
&<\int_\X \frac{  \frac{1}{1-\delta(x,y)}\sum_{i=1}^d c_i(x)p^\tau_{A_i}(y)}{\sum_{i=1}^d b_i p^\tau_{A_i}(y)} \ts |p^\tau_x(y)| \ts dy =: (\star).
\end{align*}
As for all $x\in\X$ it holds $c_i(x)\geq 0$ and $\sum_{i=1}^d c_i(x) \le 1$, we obtain
\begin{align}
\left\|\frac{\frac{1}{1-\delta(x,\cdot)}\sum_{i=1}^d c_i(x)p^\tau_{A_i}}{ \sum_{i=1}^db_i p^\tau_{A_i}} \right\|_\infty 
&\leq \frac{1}{1-\delta^*} \left\|\frac{\sum_{i=1}^d c_i(x) p^\tau_{A_i}}{\sum_{i=1}^d b_i p^\tau_{A_i}}\right\|_\infty \nonumber \\ 
&\leq \frac{1}{1-\delta^*} \left\|   \frac{ \sum_{i=1}^d p^\tau_{A_i} }{ \sum_{i=1}^d b_i p^\tau_{A_i}} \right\|_\infty \label{eq:pAi_estimeta} \\
 &\leq \frac{1}{(1-\delta^*)\min_i b_i}. \nonumber
\end{align}
With this, we can estimate the integral $(\star)$ as
\begin{equation*}
(\star) \leq \int_\X \frac{1}{(1-\delta^*)\min_i b_i} |p^\tau_x(y)|\ts dy = \frac{1}{(1-\delta^*)\min_i b_i} \underbrace{\big\| p^\tau_x \big\|_{L^1}}_{=1}. \qedhere
\end{equation*}
\end{proof}

\begin{proof}[Proof of Lemma~\ref{lem:L2 estimation}]
The proof is completely analogous to the proof of Lemma~\ref{lem:L1 estimation}, while in the estimate corresponding to~\eqref{eq:pAi_estimeta} we use that
\[
\max_i |c_i(x_1) - c_i(x_2)| \le 1. \qedhere
\]
\end{proof}

%
}

\end{document}